\documentclass{birkmult}
\newcommand{\qPunkt}{\	\ldotp}

\newcommand{\qKomma}{\	,}

\newcommand{\aeqts}{\begin{equation*}}
\newcommand{\zeqts}{\end{equation*}}

\newcommand{\aeqt}{\begin{equation}}
\newcommand{\zeqt}{\end{equation}}

\newcommand{\aenum}{\begin{enumerate}}
\newcommand{\zenum}{\end{enumerate}}

\newcommand{\eklam}[1]{\left[ \, #1 \, \right]}
\newcommand{\gklam}[1]{\left\{ \, #1 \, \right\}}
\newcommand{\rklam}[1]{\left( \, #1 \, \right)}

\newcommand{\drklam}[1]{\left( #1 \right)}

\newcommand{\Adj}[1]{#1^\ast}

\newcommand{\fref}[1]{(\ref{#1})}

\newcommand{\iEl}[2]{ #1^{\drklam{#2}} }

\newcommand{\fEl}[2]{#1_{#2}}

\newcommand{\mkEl}[3]{#1_{ #2, \,	#3 }}

\newcommand{\FKlam}[5]{ \rklam{ \fEl{#1}{#2} }_{#3 = #4}^{#5}}

\newcommand{\fKlam}[4]{ \FKlam{#1}{#2}{#2}{#3}{#4}}

\def\rddots{\mathinner{%
  \mkern1mu\raise1pt\hbox{.}%
  \mkern2mu\raise4pt\hbox{.}%
  \mkern2mu\raise7pt\vbox{\kern7pt\hbox{.}}\mkern1mu}}

\newcommand{\intervalO}[2]{ \rklam{ #1, \, #2 } }

\newcommand{\orth}[1]{ {#1}^{\perp} }
\newcommand{\wt}[1]{ \widetilde{#1} }

\newcommand{\abs}[1]{ \left|	#1	\right| }
\newcommand{\norm}[1]{ \left\lVert  #1  \right\rVert }
\newcommand{\inorm}[2]{ {\left\lVert  #1  \right\rVert}_{#2} }

\makeatletter
\renewcommand{\p@enumii}{}
\renewcommand{\p@enumiii}{}
\makeatother

\newcommand{\rklamPaar}[2]{ \rklam{ #1 \, , \	#2 } }

\newcommand{\rklamFunk}[2]{#1 \rklam{#2}}

\newcommand{\Inv}[1]{ { #1 }^{-1} }

 \newtheorem{thm}{Theorem}[section]
 \newtheorem{cor}[thm]{Corollary}
 \newtheorem{lem}[thm]{Lemma}
 \newtheorem{prop}[thm]{Proposition}
 \theoremstyle{definition}
 \newtheorem{defn}[thm]{Definition}
 \theoremstyle{remark}
 \newtheorem{rem}[thm]{Remark}
 
 \numberwithin{equation}{section}

\newcommand{\cA}{{\mathcal A}}

\newcommand{\cM}{{\mathcal M}}
\newcommand{\cC}{{\mathcal C}}
\newcommand{\cR}{{\mathcal R}}
\newcommand{\cS}{{\mathcal S}}

\newcommand{\cL}{{\mathcal L}}
\newcommand{\cP}{{\mathcal P}}
\newcommand{\cQ}{{\mathcal Q}}
\newcommand{\cU}{{\mathcal U}}
\newcommand{\gB}{{\mathfrak B}}
\newcommand{\gH}{{\mathfrak H}}
\newcommand{\gF}{{\mathfrak F}}
\newcommand{\gG}{{\mathfrak G}}
\newcommand{\gD}{{\mathfrak D}}
\newcommand{\gL}{{\mathfrak L}}

\newcommand{\gM}{{\mathfrak M}}
\newcommand{\fL}{{\mathfrak L}}
\newcommand{\C}{{\mathbb C}}
\newcommand{\D}{{\mathbb D}}

\newcommand{\N}{{\mathbb N}}
\newcommand{\R}{{\mathbb R}}
\newcommand{\T}{{\mathbb T}}
\newcommand{\Z}{{\mathbb Z}}

\newcommand{\dEins}{{\mathbf{1}}}

\newcommand{\dT}{{\mathbb T}}
\newcommand{\dC}{{\mathbb C}}

\newcommand{\Pol}{{\mathcal Pol}}
\newcommand{\Lin}{{\mathcal Lin}}
\newcommand{\sgn}{{\rm sgn\,}}
\newcommand{\Ran}{{\rm Ran\,}}
\newcommand{\ul}{\underline}

\renewcommand{\Re}{{\rm Re\,}}
\newcommand{\rank}{{\rm rank\,}}
\newcommand{\Rstr}{{\rm Rstr.\,}}

\begin{document}
%
%
%
%
%
%
%
%
%
\title[]{Description of Helson-Szeg\H o measures in terms of the Schur parameter sequences of associated Schur functions}

\author[V.K. Dubovoy]{Vladimir K. Dubovoy}
\address{%
Department of Mathematics and Mechanics\\
Kharkov State University\\
Svobody Square 4 \\
UA-61077 Kharkov, 
Ukraine}
\email{DUBOVOY@online.kharkiv.com}

\author[Fritzsche]{Bernd Fritzsche}
\address{Mathematisches Institut\\
Universit\"at Leipzig\\
Augustusplatz~10/11\\
04109~Leipzig\\
Germany}
\email{fritzsche@math.uni-leipzig.de}

\author[Kirstein]{Bernd Kirstein}
\address{Mathematisches Institut\\
Universit\"at Leipzig\\
Augustusplatz~10/11\\
04109~Leipzig\\
Germany}
\email{kirstein@math.uni-leipzig.de}



\date{\today}

\dedicatory{Dedicated to the memory of  Israel Gohberg}

\begin{abstract}
Let $\mu$ be a probability measure on the Borelian $\sigma$-algebra of the unit circle. Then we associate a Schur function
$\theta$ in the unit disk with $\mu$ and give characterizations of the case that $\mu$ is a Helson-Szeg\H o measure in terms of the sequence of Schur parameters of $\theta$.
Furthermore, we state some connections of these characterizations with the backward shift.
\end{abstract}

\maketitle

\vspace{-0.5cm}

\section{Interrelated quadruples consisting of a probability measure, a normalized Carath\'eodory function, a Schur function and a sequence of contractive complex numbers}
\label{s0}


Let $\D:=\{\zeta\in\C:|\zeta|<1\}$ and $\T:=\{t\in\C:|t|=1\}$ be the unit disk and the unit circle in the complex plane $\C$,
respectively. The central object in this paper is the class $\cM_+(\T)$ of all finite nonnegative measures on
the Borelian $\sigma$-algebra $\gB$ on $\T$. A measure $\mu\in\cM_+(\T)$ is called probabiltity measure if $\mu(\T)=1$. 
We denote by $\cM_+^1(\T)$ the subset of all probability measures which belong to $\cM_+(\T)$. 

Now we are going to introduce the subset of Helson-Szeg\H o measures on $\T$. For this reason, we denote by $\Pol$
the set of all trigonometric polynomials, i.e. the set of all functions $f:\T\rightarrow\C$ for which 
there exist a finite subset $I$ of the set $\Z$ of all integers and a sequence $(a_k)_{k\in I}$ from $\C$ such that 
\begin{align}\label{s0-1}
f(t)&=\sum_{k\in I}a_kt^k, \qquad t\in\T.
\end{align}
If $f\in\Pol$ is given via (\ref{s0-1}), then the conjugation $\tilde f$ of $f$ is defined via
\begin{align}\label{s0-2}
\tilde f(t):=-i\sum_{k\in I}(\sgn k)a_kt^k, \qquad t\in\T,
\end{align}
where $\sgn 0:=0$ and where $\sgn k:=\frac k{|k|}$ for each $k\in\Z\setminus\{0\}$.

\begin{defn}\label{s0-d1}
A non-zero measure $\mu$ which belongs to $\cM_+(\T)$ is called a \emph{Helson-Szeg\H o measure} if there exists a positive real constant $C$
such that for all $f\in\Pol$ the inequality
\begin{align}\label{s0-d1-1}
\int_\T|\tilde f(t)|^2\mu(dt)\le C\int_\T|f(t)|^2\mu(dt)
\end{align}
is satisfied.
\end{defn}

If $\mu\in\cM_+(\T)$, then $\mu$ is a Helson-Szeg\H o measure if and only if 
$\alpha\mu$ is a Helson-Szeg\H o measure for each $\alpha\in(0,+\infty)$. Thus, the investigation of 
Helson-Szeg\H o measures can be restricted to the class $\cM_+^1(\T)$.

The main goal of this paper is to describe all Helson-Szeg\H o measures $\mu$ belonging to $\cM_+^1(\T)$
in terms of the Schur parameter sequence of some Schur function $\theta$ which will be associated with $\mu$.

Let $\cC(\D)$ be the Carath\'eodory class of all functions $\Phi:\D\rightarrow\C$ which are holomorphic in $\D$ 
and which satisfy $\Re\Phi(\zeta)\ge0$ for each $\zeta\in\D$. Furthermore, let 
$$ \cC^0(\D):=\{\Phi\in\cC(\D):\Phi(0)=1\}.$$

The class $\cC(\D)$ is intimately related with the class $\cM_+(\T)$.
According to the Riesz-Herglotz theorem (see, e.g., \cite[Chapter 1]{3}), 
for each function $\Phi\in\cC(\D)$ there exist a unique measure
$\mu\in\cM_+(\T)$ and a unique number $\beta\in\R$ such that
\begin{align}\label{s0-4}
 \Phi(\zeta) &= \int_\T\frac{t+\zeta}{t-\zeta}\,\mu(dt)+i\beta,  \qquad \zeta\in\D.
\end{align}
Obviously, $\beta = {\rm Im}\, [\Phi (0)]$. On the other hand, it can be easily checked that, for arbitrary
$\mu\in\cM_+(\T)$ and $\beta \in\R$, the function $\Phi$ which is defined by the
right-hand side of (\ref{s0-4}) belongs to $\cC(\D)$. 
If we consider the Riesz-Herglotz representation (\ref{s0-4}) for a function $\Phi\in\cC^0(\D)$,
then $\beta =0$ and $\mu$ belongs to the set $\cM_+^1 (\mathbb T)$.
Actually, in this way we obtain a bijective correspondence between the classes $\cC^0(\D)$ and $\cM_+^1(\T)$.


Let us now consider the Schur class $\cS(\D)$ of all functions $\Theta :\D\to\C$ 
which are holomorphic in $\D$ and which satisfy $\Theta(\D)\subseteq \D\cup\T$.
If $\Theta\in\cS (\mathbb D)$, then the function $\Phi:\mathbb D\to\mathbb C$ defined by
\begin{equation}\label{Nr.0.1}
\Phi (\zeta) := \frac{1+\zeta\Theta (\zeta)}{1-\zeta\Theta (\zeta)}
\end{equation}
belongs to the class $\cC^0(\D)$. Note that from \eqref{Nr.0.1} it follows
\begin{equation}\label{Nr.0.1f}
 \zeta\Theta (\zeta) = \frac{\Phi (\zeta) - 1}{\Phi (\zeta) + 1}\,, \quad \zeta\in\mathbb D.
\end{equation}
Consequently, it can be easily verified that via \eqref{Nr.0.1} a bijective correspondence between the classes $\cS (\mathbb D)$
and $\cC^0 (\mathbb D)$ is established. 


Let $\theta\in\cS$. Following I. Schur \cite{4}, we set $\theta_0:=\theta$ and $\gamma_0:=\theta_0(0)$. Obviously, $|\gamma_0|\leq 1$. If $|\gamma_0|< 1$, then we consider the function $\theta_1:\D\rightarrow\C$ defined by
$$\theta_1(\zeta):=\frac{1}{\zeta}\cdot\frac{\theta_0(\zeta)-\gamma_0}{1-\overline{\gamma_0}\theta_0(\zeta)}\,.$$
In view of the Lemma of H.A. Schwarz, we have $\theta_1\in\cS$. As above we set $\gamma_1:=\theta_1(0)$ and if $|\gamma_1|<1$, we consider the function $\theta_2:\D\rightarrow\C$ defined by
$$\theta_2(\zeta):=\frac{1}{\zeta}\cdot\frac{\theta_1(\zeta)-\gamma_1}{1-\overline{\gamma_1}\theta_1(\zeta)}\,.$$
Further, we continue this procedure inductively. Namely, if in the $j$-th step a function $\theta_j$ occurs for which 
the complex number $\gamma_j:=\theta_j(0)$ fulfills
$|\gamma_j|<1$, we define $\theta_{j+1}:\D\rightarrow\C$ by
$$\theta_{j+1}(\zeta):=\frac{1}{\zeta}\cdot\frac{\theta_j(\zeta)-\gamma_j}{1-\overline{\gamma_j}\theta_j(\zeta)}$$
and continue this procedure in the prescribed way. Let $\N_0$ be the set of all nonnegative integers,
and, for each $\alpha\in\R$ and $\beta\in\R\cup\{+\infty\}$, let
$\N_{\alpha,\beta}:=\{k\in\N_0:\alpha\le k\le\beta\}$. Then two cases are possible:
\begin{enumerate}
 \item[(1)] The procedure can be carried out without end, i.e. $|\gamma_j|<1$ for each $j\in\N_0$.
\item[(2)] There exists an $m\in\N_0$ such that $|\gamma_m|=1$ and, if $m>0$, then $|\gamma_j|<1$ for each $j\in\N_{0,m-1}$.
\end{enumerate}
Thus, for each function $\theta\in\cS$, a sequence $(\gamma_j)_{j=0}^\omega$ is associated with $\theta$. Here we have $\omega=\infty$ (resp. $\omega=m$) in the first (resp. second) case. From I. Schur's paper \cite{4} it is known that the second case occurs if and only if $\theta$ is a finite Blaschke product of degree $\omega$. 

The above procedure is called \emph{Schur algorithm} and the sequence $(\gamma_j)_{j=0}^\omega$ obtained here is called 
\emph{the sequence of Schur parameters associated with the function $\theta$}, whereas for each $j\in\N_{0,\omega}$ the function $\theta_j$ is called \emph{the $j$-th Schur transform of $\theta$}. The symbol $\Gamma$ stands for the set of all sequences of Schur parameters associated with functions belonging to $\cS$.\\
The following two properties established by I. Schur in \cite{4} determine the particular role which Schur parameters play in the study of functions of class $\cS$.
\begin{enumerate}
 \item[(a)] Each sequence $(\gamma_j)_{j=0}^\infty$ of complex numbers with $|\gamma_j|<1$ for each $j\in\N_0$ belongs to $\Gamma$.
 Furthermore, for each $n\in\N_0$, a sequence $(\gamma_j)_{j=0}^n$ of complex numbers with $|\gamma_n|=1$ and  $|\gamma_j|<1$ for each $j\in\N_{0,n-1}$ belongs to $\Gamma$.
\item[(b)] There is a one-to-one correspondence between the sets $\cS$ and $\Gamma$.
\end{enumerate}
Thus, the Schur parameters are independent parameters which completely determine the functions of the class $\cS$.

In the result of the above considerations we obtain special ordered qudruples $[\mu,\Phi,\Theta,\gamma]$ consisting
of a measure $\mu\in\cM_+^1(\T)$, a function $\Phi\in\cC^0(\D)$, a function $\Theta\in\cS(\D)$, 
and Schur parameters $\gamma=(\gamma_j)_{j=0}^\omega\in\Gamma$,
which are interrelated in such way that each of these four objects uniquely
determines the other three ones. For that reason, if one of the four objects is given, we will call the three others
associated with it.


The main goal of this paper is to derive a criterion which gives an answer to the question when a measure 
$\mu\in\cM_+^1(\T)$ is a Helson-Szeg\H o measure (see Section \ref{s6}).
For this reason, we will need the properties of Helson-Szeg\H o measures listed below (see Theorem \ref{s0-t1}).
For more information about Helson-Szeg\H o measures, we refer the reader, e.g., to \cite[Chapter 7]{2}, \cite[Chapter 5]{5}.

Let $f\in\Pol$ be given by (\ref{s0-1}). Then we consider the Riesz projection $P_+f$ which is defined by
$$ (P_+f)(t):=\sum_{k\in I\cap\N_0}a_kt^k, \qquad t\in\T. $$
Let $\Pol_+:=\Lin\{t^k:k\in\N_0\}$ and $\Pol_-:=\Lin\{t^{-k}:k\in\N\}$ where $\N$ is the set of all positive integers.
Then, clearly, $P_+$ is the projection which projects the linear space $\Pol$ onto $\Pol_+$ parallel to $\Pol_-$.

In view of a result due to Fatou (see, e.g., \cite[Theorem 1.18]{3}), we will use the following notation:
If $h\in H^2(\D)$, then the symbol $\ul h$ stands for the radial boundary values of $h$, which exist for $m$-a.e. $t\in\T$ where $m$ is the normalized Lebesgue measure on $\T$.
If $z\in\C$, then the symbol $z^*$ stands for the complex conjugate of $z$.

\begin{thm}\label{s0-t1}
Let $\mu\in\cM_+^1(\T)$. Then the following statements are equivalent:
\begin{enumerate}
\item[(i)] $\mu$ is a Helson-Szeg\H o measure.
\item[(ii)] The Riesz projection $P_+$ is bounded in $L_\mu^2$.
\item[(iii)] The sequence $(t^n)_{n\in\Z}$ is a (symmetric or nonsymmetric) basis of $L_\mu^2$.
\item[(iv)] $\mu$ is absolutely continuous with respect to $m$ and there is an outer function
$h\in H^2(\D)$ such that $\frac{d\mu}{dm}=|\ul h|^2$ and 
$$ {\rm dist\,}\big(\ul h^*\!/\ul h,H^\infty(\T)\big)<1. $$
\end{enumerate}
\end{thm}

\begin{cor}\label{s0-c1}
Let $\mu\in\cM_+^1(\T)$ be a Helson-Szeg\H o measure. 
Then the Schur parameter sequence $\gamma=(\gamma_j)_{j=0}^\omega$ associated 
with $\mu$ is infinite, i.e., $\omega=\infty$ holds, and the series $\sum_{j=0}^\infty|\gamma_j|^2$ converges, i.e. 
$\gamma\in l_2$.
\end{cor}

\begin{proof}
Let $\theta\in\cS(\D)$ be the Schur function associated with $\mu$. Then it is known (see, e.g. \cite[Chapter 3]{6})
that 
\begin{align}\label{s0-c1-1}
\prod_{j=0}^\omega(1-|\gamma_j|^2)&=\exp\Big\{\int_\T\ln(1-|\ul\theta(t)|^2)m(dt)\Big\}.
\end{align}
We denote by $\Phi$ the function from $\cC^0(\D)$ associated with
$\mu$. Using (\ref{s0-4}), assumption (iv) in Theorem \ref{s0-t1}, and Fatou's theorem we obtain
\begin{align}\label{s0-c1-2}
1-|\ul\theta(t)|^2 = \frac{4\Re\ul\Phi(t)}{|\ul\Phi(t)+1|^2} = \frac{4|\ul h(t)|^2}{|\ul\Phi(t)+1|^2} 
\end{align}
for $m$-a.e. $t\in\T$. Thus,
\begin{align}\label{s0-c1-3}
\ln(1-|\ul\theta(t)|^2) = \ln 4+2\ln|\ul h(t)|-2\ln|\ul\Phi(t)+1|. 
\end{align}
In view of $\Re\Phi(\zeta)>0$ for each $\zeta\in\D$, the function $\Phi$ is outer. Hence, $\ln|\ul\Phi+1|\in L_m^1$ (see, e.g. 
\cite[Theorem 4.29 and Theorem 4.10]{3}). Taking into account condition (iv) in Theorem \ref{s0-t1},
we obtain $h\in H^2(\D)$. Thus,
we infer from (\ref{s0-c1-3}) that $\ln(1-|\ul\theta(t)|^2)\in L_m^1$. Now the assertion follows
from (\ref{s0-c1-1}).
\end{proof}

\begin{rem}\label{s0-r1}
Let $\mu\in\cM_+^1(\T)$, and let the Lebesgue decomposition of $\mu$ with respect to $m$ be given by
\begin{align}\label{s0-r1-1}
\mu(dt)&=v(t)m(dt)+\mu_s(dt),
\end{align}
where $\mu_s$ stands for the singular part of $\mu$ with repect to $m$. Then the relation $\Re\ul\Phi=v$ holds
$m$-a.e. on $\T$. The identity (\ref{s0-c1-3}) has now the form
$$ \ln(1-|\ul\theta(t)|^2) = \ln 4+\ln v(t)-2\ln|\ul\Phi(t)+1| $$
for $m$-a.e. $t\in\T$. From this and (\ref{s0-c1-1}) now it follows a well-known result, namely, that $\ln v\in L_m^1$
(i.e., $\mu$ is a Szeg\H o measure) if and only if $\omega=\infty$ and $\gamma\in l_2$.
In particular, a Helson-Szeg\H o measure is also a Szeg\H o measure.
\end{rem}

We note that our interest in describing the class of Helson-Szeg\H o measures in terms of Schur parameters was 
initiated by conversations with L. B. Golinskii and A. Ya. Kheifets.
During these discussions they turned our attention to the importance of this description for the solution of the inverse scattering problem (the heart of Faddeev-Marchenko theory),
that is, to give necessary and sufficient conditions on a certain class 
of CMV matrices such that the restriction of this correspondence
is one to one (see \cite{GKPY} and the related papers \cite{KPY} and \cite{PVY}).
Our approach to the description of Helson-Szeg\H o measures differs from the one in \cite{GKPY} in that we investigate
this question for CMV matrices in another basis
(see \cite[Definition 2.2., Theorem 2.13]{7}), 
namely the one for that CMV matrices have the full GGT representation
(see Simon \cite[p. 261--262, Remarks and Historical Notes]{Sim}).
However, we consider the study of interrelations between Helson-Szeg\H o
measures and Schur parameter sequences to be of particular interest on its own.
The scope of these investigations would extend far beyond inverse scattering.


\section{A unitary collogation associated with a Borelian probability measeure on the unit circle}
\label{s2}

The starting point of this section is the observation that a given Schur function $\Theta\in\cS(\mathbb D)$ can be
represented as characteristic function of some contraction in a Hilbert space. That means that there exists a separable
complex Hilbert space $\gH$ and bounded linear operators $T:\gH\to\gH$,\, $F:\C\to\gH$,\, $G:\gH\to\C$, and $S:\C\to\C$ 
such that the block operator 
\begin{equation}\label{Nr.d1.2b}
 U:= \left(\begin{matrix} T & F \cr G & S \end{matrix}\right) :\gH\oplus\C\to\gH\oplus\C
\end{equation}
is unitary and, moreover, that for each $\zeta\in\mathbb D$ the equality
\begin{equation}\label{Nr.d1.2c}
 \Theta (\zeta) = S + \zeta G (I - \zeta T)^{-1} F,
\end{equation}
is fulfilled. Note that in \eqref{Nr.d1.2b} the complex plane $\C$ is considered as the one-dimensional
complex Hilbert space with the usual inner product
\[ \bigl(z,w\bigr)_{\C} = z^{\ast}w,
\quad z,w\in\C.   \]

The unitarity of the operator $U$ implies that the operator $T$ is contractive (i.e. $\|T\|\leq 1$).
Thus, for all $\zeta\in\D$ the operator $I - \zeta T$ is boundedly invertible. 
The unitarity of the operator $U$ means that the ordered tuple
\begin{equation}\label{Nr.d1.5b}
 \bigtriangleup := (\gH, \C, \C; T, F, G, S)
\end{equation}
is a unitary colligation. In view of \eqref{Nr.d1.2c}, the function $\Theta$ is the characteristic
operator function of the unitary colligation $\bigtriangleup$. 
For a detailed treatment of unitary colligations and their characteristic functions we refer the reader to the landmark paper 
\cite{8}.

The following subspaces of $\gH$ will play an important role in the sequel
\begin{align}\label{Nr.d1.5c} 
 & \gH_{\gF}:=\bigvee\limits_{n=0}^{\infty}T^nF(\C),\quad
 \gH_{\gG}:=\bigvee\limits_{n=0}^{\infty}(T^{*})^{n} G^*(\C) .
\end{align}
By the symbol $\bigvee_{n=0}^{\infty} A_n$ we mean the smallest closed subspace 
generated by the subsets $A_n$ of the corresponding spaces. 
The subspaces $\gH_\gF$ and $\gH_\gG$ are called the subspaces of controllability and
observability, respectively.
We note that the unitary operator $U$ can be chosen such that 
\begin{equation}\label{Nr.d1.5d} 
  \gH = \gH_\gF\vee\gH_\gG 
\end{equation}
holds. In this case the unitary colligation $\bigtriangleup$ is called simple. 
The simplicity of a unitary colligation means that there does not exist a nontrivial 
invariant subspace of $\gH$ on which the operator $T$ induces a unitary operator.
Such kind of contractions $T$ are called completely nonunitary.

\begin{prop}\label{s2-p1}
Let $\mu\in\cM_+^1(\T)$ be a Szeg\H o measure. Let $\Theta$ be the function belonging to $\cS(\D)$ which is associated with $\mu$ and let $\bigtriangleup$ be the simple unitary colligation the characteristic operator function of which coincides with $\Theta$. Then the spaces 
\begin{align}\label{s2-p1-1}
& \gH_\gF^\bot:=\gH\ominus\gH_\gF, \qquad \gH_\gG^\bot:=\gH\ominus\gH_\gG 
\end{align}
are nontrivial.
\end{prop}

\begin{proof}
Let $\gamma=(\gamma_j)_{j=0}^\omega\in\Gamma$ be the Schur parameter sequence of $\Theta$. Then from Corollary \ref{s0-c1}
we infer that $\omega=\infty$ and that $\gamma\in l_2$. In this case it was proved in \cite[Chapter 2]{7} that both spaces 
(\ref{s2-p1-1}) are nontrivial.
\end{proof}

Because of \eqref{Nr.d1.5c} and \eqref{s2-p1-1} it follows that the subspace $\gH_\gG^\perp$ 
$($resp. $\gH_\gF^\perp)$ is invariant with respect to $T$ $($resp. $T^*)$.
It can be shown (see \cite[Theorem 1.6]{7}) that 
$$  V_T:= {\rm Rstr}._{\gH_\gG^\perp}T \quad\mbox{and}\quad  V_{T^*}:= {\rm Rstr.}_{\gH_\gF^\perp}T^* $$
are both unilateral shifts. More precisely, $V_T$ (resp. $V_{T^*}$) is exactly the maximal unilateral shift
contained in $T$ $($resp. $T^*)$. This means that an arbitrary invariant subspace with
respect to $T$ $($resp. $T^*)$ on which $T$ $($resp. $T^*)$ induces a unilateral shift is 
contained in $\gH_\gG^\perp$ $($resp. $\gH_\gF^\perp)$.

Let $\mu\in\cM^1_+(\T)$. Then our subsequent
considerations are concerned with the investigation of the unitary
operator $U_\mu^\times:L_\mu^2\to L_\mu^2$ which is defined for each $f\in L_\mu^2$ by
\begin{align}\label{Nr.2.0}
 (U_\mu^\times f) (t) &:= t^* \cdot f(t),
 \qquad t\in\mathbb T.
\end{align}
Denote by $\tau$ the embedding operator of $\C$ into $L_\mu^2$, i.e., $\tau :\C\to L_\mu^2$
is such that for each $c\in\C$ the image $\tau (c)$ of $c$ is the constant function on $\T$
with value $c$. Denote by $\mathbb C_\T$ the subspace of $L_\mu^2$ which is generated by
the constant functions and denote by ${\bf 1}$ the constant function on $\T$ with value $1$.
Then obviously $\tau (\C) =\C_\T$ and $\tau (1) = {\bf 1}$.

We consider the subspace
$$ \gH_\mu := L_\mu^2 \ominus \mathbb C_\mathbb T. $$
Denote by
$$ U_\mu^\times = \left(\begin{matrix} T^\times & F^\times\cr G^\times & S^\times  \end{matrix}\right) $$
the block representation of the operator $U_\mu^\times$ with respect to the orthogonal decomposition
$L_\mu^2 = \gH_\mu\oplus \mathbb C_\mathbb T$. 
Then (see \cite[Section 2.8]{7}) the following result holds.

\begin{thm}\label{2.1}
Let $\mu\in\cM_+^1 (\mathbb T)$. Define $T_\mu := T^\times$,\; $F_\mu := F^\times \tau$,\;
$G_\mu := \tau^\ast G^\times$,\;and $S_\mu := \tau^\ast S^\times\tau$. Then
\begin{align}\label{Nr.2.2}
 & \bigtriangleup_\mu := ( \gH_\mu,\mathbb C, \mathbb C;  T_\mu, F_\mu, G_\mu, S_\mu)
\end{align}
is a simple unitary colligation the characteristic function $\Theta_{\bigtriangleup_\mu}$ of which
coincides with the Schur function $\Theta$ associated with $\mu$.
\end{thm}

In view of Theorem~\ref{2.1}, the operator $T_\mu$ is a completely nonunitary contraction and if the function $\Phi$ is given by 
(\ref{s0-4}) with $\beta=0$, then from (\ref{Nr.0.1f}) it follows
$$ \zeta\Theta_{\bigtriangleup_\mu} (\zeta) = \frac{\Phi (\zeta) -1}{\Phi (\zeta) + 1},
\quad\zeta\in\mathbb D. $$

\begin{defn}
 Let $\mu\in\cM_+^1 (\T)$. Then the simple unitary colligation given by \eqref{Nr.2.2} is called 
 \emph{the unitary colligation associated with} $\mu$.
\end{defn}

Let $\mu\in\cM_+^1(\T)$ be a Szeg\H o measure and let $\gamma=(\gamma_j)_{j=0}^\omega\in\Gamma$ be the Schur parameter 
sequence associated with $\mu$. Then Remark \ref{s0-r1} shows that $\omega=\infty$ and $\gamma\in l_2$.
Furthermore, we use for all integers $n$ the setting $e_n :\mathbb T\to\mathbb C$ defined by
\begin{align}\label{Nr.EK}
 & e_n (t) := t^n .
\end{align}
Thus, we have $e_{-n} = (U_\mu^\times)^n {\bf 1}$,
where $U_\mu^\times$ is the operator defined by (\ref{Nr.2.0}). 
We consider then the system $\{e_0, e_{-1}, e_{-2}, \ldots\}$. 
By the Gram-Schmidt orthogonalization method in the space $L_\mu^2$ we get a unique sequence
$(\varphi_n)_{n=0}^\infty$ of polynomials, where 
\begin{align}\label{Nr.3.1}
 & \varphi_n (t) = \alpha_{n,n} t^{-n} + \alpha_{n,n-1} t^{-(n-1)} + \cdots + \alpha_{n,0},
 \; t\in\mathbb T,\quad n\in \N_0,
\end{align}
such that the conditions
\begin{align}\label{Nr.3.2}
 &\bigvee_{k=0}^n \varphi_k = \bigvee_{k=0}^n (U_\mu^\times)^k {\bf 1},
 \quad \bigl((U_\mu^\times)^n {\bf 1}, \varphi_n\bigr)_{L_\mu^2} > 0, 
 \quad n\in \N_0,
\end{align}
are satisfied. 
We note that the second condition in (\ref{Nr.3.2}) is equivalent to \linebreak
$\bigl( {\bf 1}, \varphi_0 \bigr)_{L_\mu^2} > 0$ and 
\begin{align}\label{Nr.3.3}
 & \bigl(U_\mu^\times \varphi_{n-1},\varphi_n\bigr)_{L_\mu^2} > 0, 
 \quad n\in \N_0.
\end{align}
In particular, since $\mu (\mathbb T)=1$ holds, from the construction of $\varphi_0$ we see that
\begin{align}\label{Nr.4.11}
  & \varphi_0 = \boldsymbol{1}.
\end{align}
We consider a simple unitary colligation $\bigtriangleup_\mu$ of the type \eqref{Nr.2.2} associated with the measure $\mu$. 
The controllability and observability spaces \eqref{Nr.d1.5c} associated with the unitary colligation $\bigtriangleup_\mu$ 
have the forms
\begin{align}\label{s2-17} 
 & \gH_{\mu,\gF}=\bigvee_{n=0}^\infty T_\mu^nF_\mu(\C) \quad\mbox{and}\quad
 \gH_{\mu,\gG}=\bigvee_{n=0}^\infty(T_\mu^*)^n G_\mu^*(\C) ,
\end{align}
respectively.
Let the sequence of functions $(\varphi'_k)_{k=1}^\infty$ be defined by
\begin{align}\label{Nr.3.9}
 & \varphi'_k := T_\mu^{k-1}  F_\mu (1), \quad k\in\N.
\end{align}
In view of the formulas
\begin{align}\label{Nr.3.8}
 & \bigvee_{k=0}^n (U_\mu^\times)^k {\bf 1} 
 = \left( \bigvee_{k=0}^{n-1} (T_\mu)^k F_\mu (1)\right) \oplus \mathbb C_\mathbb T, 
 \quad n\in\N,
\end{align}
it can be seen that the sequence $(\varphi_k)_{k=1}^\infty$ can
be obtained by applying the Gram-Schmidt orthonormalization procedure to 
$(\varphi'_k)_{k=1}^\infty$ with additional consideration of the normalization 
condition (\ref{Nr.3.3}). Thus, we obtain the following result:

\begin{thm}\label{s2-t2}
 The system $(\varphi_k)_{k=1}^\infty$ of orthonormal polynomials is a basis in the space 
 $\gH_{\mu, \gF}$, and 
 \begin{align}\label{s2-t2-1}
 & \gH_{\mu,\gF}=\left(\bigvee_{k=0}^\infty (t^*)^k\right)\ominus\C_\T.
 \end{align}
 This system can be obtained in the result of the application of the Gram-Schmidt
 orthogonalization procedure to the sequence \eqref{Nr.3.9} taking into account the normalization 
 condition \eqref{Nr.3.3}.
\end{thm}

\begin{rem}\label{s2-r2}
 Analogously to \eqref{s2-t2-1} we have the equation
  \begin{align}\label{s2-r2-1}
 & \gH_{\mu,\gG}=\left(\bigvee_{k=0}^\infty t^k\right)\ominus\C_\T.
 \end{align}
\end{rem}

If $T$ is a contraction acting on some Hilbert space $\gH$, then we use the setting
$$ \delta_T:=\dim {\gD_T} \quad 
   \bigl(\mbox{resp.}\; \delta_{T^*}:=\dim {\gD_{T^*}}  \bigr) , $$
where $\gD_T:=\overline{D_T(\gH)}$ (resp. $\gD_{T^*}:=\overline{D_{T^*}(\gH)}_{\,}$) 
is the closure of the range of the defect operator $D_T:=\sqrt{I_\gH-T^*T}$
(resp. $D_{T^*}:=\sqrt{I_\gH-TT^*}_{\,}$).

In view of \eqref{s2-p1-1}, let
$$ \gH_{\mu,\gF}^\perp := \gH_\mu\ominus \gH_{\mu,\gF} , \qquad \gH_{\mu,\gG}^\perp := \gH_\mu\ominus \gH_{\mu,\gG}. $$ 
If $\mu$ is a Szeg\H o measure, then we have
\begin{align}\label{s2-15}
 L_\mu^2\ominus\bigvee_{k=0}^\infty e_{-k} 
 =\big(\gH_\mu\oplus\C_\T\big)\ominus\big(\gH_{\mu,\gF}\oplus\C_\T\big)
 = \gH_\mu\ominus \gH_{\mu,\gF} = \gH_{\mu,\gF}^\perp \ne\{0\}.
\end{align}
So from Proposition \ref{s2-p1} we obtain the known result that in this case the system $(\varphi_n)_{n=0}^\infty$ is not complete in the space $L_\mu^2$.

In our case 
$$  V_{T_\mu}:= {\rm Rstr.}_{\gH_{\mu,\gF}^\perp}T_\mu \qquad
  \big(\mbox{resp.}\  V_{T_\mu^*}:= {\rm Rstr.}_{\gH_{\mu,\gF}^\perp}T_\mu^* \big) $$
is the maximal unilateral shift contained in $T_\mu$ (resp. $T_\mu^*$) (see \cite[Theorem~1.6]{7}). 
In view of $\delta_{T_\mu} = \delta_{T_\mu^*} =1$ 
the multiplicity of the unilateral shift $V_{T_\mu}$ (resp. $V_{T_\mu^*}$) is equal to $1$.

\begin{prop}\label{3.4}
 The orthonormal system of the polynomials $(\varphi_k)_{k=0}^\infty$ is noncomplete in $L_\mu^2$
 if and only if the contraction $T_\mu$ $($resp. $T_\mu^*)$ contains a maximal unilateral 
 shift $V_{T_\mu}$ $($resp. $V_{T_\mu^*})$ of multiplicity $1$.
\end{prop}

\section{On the connection between the Riesz projection
          $\fEl{P}{+}$ and the projection $\mkEl{ \cP^{\gF} }{\mu}{\gG}$
          which projects $\fEl{\gH}{\mu}$ onto $\mkEl{\gH}{\mu}{\gG}$
          parallel to $\mkEl{\gH}{\mu}{\gF}$}
\label{s3}


Let $\mu \in \rklamFunk{ \fEl{\cM^{1}}{+} }{ \dT }$.
We consider the unitary colligation $\fEl{\Delta}{\mu}$
of type \eqref{Nr.2.2} which is associated with the measure $\mu$.
Then the following statement is true.

\begin{thm}       \label{3-thm1}
  Let $\mu \in \rklamFunk{ \fEl{\cM^{1}}{+} }{ \dT }$ be
  a Szeg\H{o} measure. Then the Riesz projection $\fEl{P}{+}$
  is bounded in $\fEl{L^2}{\mu}$ if and only if the projection
  $\mkEl{ \cP^{\gF} }{\mu}{\gG}$ which projects
  $\fEl{\gH}{\mu}$ onto $\mkEl{\gH}{\mu}{\gG}$
  parallel to $\mkEl{\gH}{\mu}{\gF}$ is bounded.
\end{thm}

\begin{proof}
  For each $n \in \N_0$ we consider particular subspaces of the space $\gH_\mu$, namely
  \begin{align}       \label{3-thm1-1}
    \mkEl{ \iEl{\gH}{n} }{\mu}{\gF}
      :=  \bigvee_{k = 0}^{n}
          { 
            \fEl{T^k}{\mu}
            \rklamFunk{ \fEl{F}{\mu} }{\dC} 
          } \qKomma
    \qquad
    \qquad
    \mkEl{ \iEl{\gH}{n} }{\mu}{\gG}
      :=  \bigvee_{k = 0}^{n}
          { 
            \rklam{ \fEl{ \Adj{T} }{\mu} }^k
            \rklamFunk{ \fEl{ \Adj{G} }{\mu} }{\dC} 
          } \qKomma
  \end{align}
  and
  \begin{align}       \label{3-thm1-2}
    \fEl{ \iEl{\gH}{n} }{\mu}
      :=  \mkEl{ \iEl{\gH}{n} }{\mu}{\gF}
          \vee
          \mkEl{ \iEl{\gH}{n} }{\mu}{\gG} \qKomma
    \qquad
    \qquad
    \mkEl{L}{\mu}{n}
      :=  \fEl{ \iEl{\gH}{n} }{\mu}
          \oplus
          \fEl{\dC}{\dT} \qPunkt
  \end{align}
  Then from \eqref{Nr.3.9}, \eqref{s2-t2-1}, and \eqref{s2-r2-1}  
  we obtain the relations
  \begin{align}       \label{3-thm1-3}
    \mkEl{\gH}{\mu}{\gF}
      =  \bigvee_{n = 0}^{\infty}
          { 
            \mkEl{ \iEl{\gH}{n} }{\mu}{\gF}
          } \qKomma
    \qquad
    \qquad
    \mkEl{\gH}{\mu}{\gG}
      =  \bigvee_{n = 0}^{\infty}
          { 
            \mkEl{ \iEl{\gH}{n} }{\mu}{\gG}
          } \qKomma
  \end{align}
  \begin{align}       \label{3-thm1-4}
    \fEl{\gH}{\mu}
      =  \bigvee_{n = 0}^{\infty}
          { 
            \fEl{ \iEl{\gH}{n} }{\mu}
          } \qKomma
    \qquad
    \qquad
    \fEl{L^2}{\mu}
      =  \bigvee_{n = 0}^{\infty}
          { 
            \mkEl{L}{\mu}{n}
          } \qKomma
  \end{align}
  \begin{align}       \label{3-thm1-5}
    \mkEl{ \iEl{\gH}{n} }{\mu}{\gF}
      =  \rklam{
            \bigvee_{k = 0}^{n}
            { 
              (t^*)^k
            }
          }
          \ominus
          \fEl{\dC}{\dT} \qKomma
    \qquad
    \qquad
    \mkEl{ \iEl{\gH}{n} }{\mu}{\gG}
      =  \rklam{
            \bigvee_{k = 0}^{n}
            { 
              t^k
            }
          }
          \ominus
          \fEl{\dC}{\dT}\,,
  \end{align}
  and
  \begin{align}       \label{3-thm1-6}
    \fEl{ \iEl{\gH}{n} }{\mu}
      =  \rklam{
            \bigvee_{k = -n}^{n}
            { 
              t^{k}
            }
          }
          \ominus
          \fEl{\dC}{\dT} \qKomma
    \qquad
    \qquad
    \mkEl{L}{\mu}{n}
      =   \bigvee_{k = -n}^{n}
            { 
              t^{k}
            } .
  \end{align}
  Since $\mu$ is a Szeg\H{o} measure, for each $n \in \N_0$
  we obtain
  \begin{align}       \label{3-thm1-7}
    \fEl{ \iEl{\gH}{n} }{\mu}
      = \mkEl{ \iEl{\gH}{n} }{\mu}{\gF}
        \quad \dot{+} \quad
        \mkEl{ \iEl{\gH}{n} }{\mu}{\gG}        \qPunkt
  \end{align}
  Suppose now that the Riesz projection $\fEl{P}{+}$ is bounded
  in $\fEl{L^2}{\mu}$. Let $h \in \fEl{ \iEl{\gH}{n} }{\mu}$.
  Then, because of \fref{3-thm1-6}, the function $h$ has the form
  \begin{align}       \label{3-thm1-8}
    \rklamFunk{h}{t}
      = \sum_{k = -n}^{n}{ \fEl{a}{k} t^k }        \,,\qquad t\in\T.
  \end{align}
  From \fref{3-thm1-5} and \fref{3-thm1-7} we obtain
  \begin{align}       \label{3-thm1-9}
    h = \fEl{h}{\gF} + \fEl{h}{\gG}   \qKomma
  \end{align}
  where
  \begin{align*}
    \fEl{h}{\gF}  \in \mkEl{ \iEl{\gH}{n} }{\mu}{\gF} \qKomma
    \qquad
    \qquad
    \fEl{h}{\gG}  \in \mkEl{ \iEl{\gH}{n} }{\mu}{\gG} \qKomma
  \end{align*}
  \begin{align}       \label{3-thm1-10}
    \rklamFunk{ \fEl{h}{\gF} }{t}
      = \mkEl{a}{0}{\gF} 
        + \sum_{k = 1}^{n}
          {
           \fEl{a}{-k} t^{-k}
          }        \qKomma
    \qquad
    \qquad
    \rklamFunk{ \fEl{h}{\gG} }{t}
      = \mkEl{a}{0}{\gG} 
        + \sum_{k = 1}^{n}
          {
           \fEl{a}{k} t^{k}
          }\,,
  \end{align}
  and
  \begin{align*}
    \fEl{a}{0} 
      = \mkEl{a}{0}{\gF} +  \mkEl{a}{0}{\gG}  \qPunkt
  \end{align*}
  Observe that
  \begin{align}       \label{3-thm1-11}
    \mkEl{ \cP^{\gF} }{\mu}{\gG}  h
    = \fEl{h}{\gG}  \qPunkt
  \end{align}
  On the other hand, we have
  \begin{align}       \label{3-thm1-12}
    h = \fEl{h}{+}  +  \fEl{h}{-} \qKomma
  \end{align}
  where, for each $t\in\T$,
  \begin{align}       \label{3-thm1-13}
    \rklamFunk{ \fEl{h}{+} }{t}
      = \rklamFunk{ \rklam{ \fEl{P}{+}(h) } }{t}
      = \sum_{k = 0}^{n}
        {
          \fEl{a}{k} t^{k}
        }        \qKomma
    \qquad
    \qquad
    \rklamFunk{ \fEl{h}{-} }{t}
      = \sum_{k = 1}^{n}
        {
         \fEl{a}{-k} t^{-k}
        } \qPunkt
  \end{align}
  For a polynomial $\fEl{h}{\gF}$ of the type \fref{3-thm1-10}
  we set
  \begin{align}       \label{3-thm1-14}
    \rklamFunk{ \fEl{h}{\gF} }{0}
      := \mkEl{a}{0}{\gF}    \qPunkt
  \end{align}
  Then from \fref{3-thm1-10}--\fref{3-thm1-14} we infer
  \begin{align}       \label{3-thm1-15}
    \fEl{P}{+}  h
    = \fEl{h}{+}
    = \fEl{h}{\gG}
      + \rklamFunk{ \fEl{h}{\gF} }{0} \cdot  \dEins
    = \mkEl{ \cP^{\gF} }{\mu}{\gG} h
      + \rklamFunk{ \fEl{h}{\gF} }{0} \cdot  \dEins  \qPunkt
  \end{align}
  Observe that
  \begin{align}       \label{3-thm1-16}
    \fEl{h}{-}
    = \fEl{h}{\gF}
      - \rklamFunk{ \fEl{h}{\gF} }{0} \cdot  \dEins \qKomma
  \end{align}
  where in view of \fref{3-thm1-5} we see that
  \begin{align*}
    \fEl{h}{\gF}  \perp \dEins \qPunkt
  \end{align*}
  Let $\fEl{P}{ \fEl{\dC}{\dT} }$ be the orthoprojection from
  $\fEl{L^2}{\mu}$ onto $\fEl{\dC}{\dT}$. Then from \fref{3-thm1-16}
  it follows that
  \begin{align*}
    \rklamFunk{ \fEl{h}{\gF} }{0} \cdot \dEins
    = \fEl{P}{ \fEl{\dC}{\dT} }  \fEl{h}{-}
    = \fEl{P}{ \fEl{\dC}{\dT} }
      \rklam{
        I - \fEl{P}{+}
      } h   \qPunkt
  \end{align*}
  Inserting this expression into \fref{3-thm1-15} we get
  \begin{align*}
    \mkEl{ \cP^{\gF} }{\mu}{\gG} h
    = \fEl{P}{+}  h
      - \fEl{P}{ \fEl{\dC}{\dT} }
        \rklam{
          I - \fEl{P}{+}
        } h  \qPunkt
  \end{align*}
  From this and \fref{3-thm1-4} it follows that the boundedness
  of the projection $\fEl{P}{+}$ in $\fEl{L^2}{\mu}$ implies
  the boundedness of the projection $\mkEl{ \cP^{\gF} }{\mu}{\gG}$
  in $\fEl{\gH}{\mu}$. 
  
  Conversely, suppose that the projection
  $\mkEl{ \cP^{\gF} }{\mu}{\gG}$ is bounded in $\fEl{\gH}{\mu}$.
  Let $f \in \mkEl{L}{\mu}{n}$. We denote by
  $\fEl{P}{ \fEl{ \iEl{\gH}{n} }{\mu} }$ the orthogonal projection
  from $\fEl{L^2}{\mu}$ onto $\fEl{ \iEl{\gH}{n} }{\mu}$. We set
  \begin{align*}
    h := \fEl{P}{ \fEl{ \iEl{\gH}{n} }{\mu} } f
  \end{align*}
  and use for $h$ the notations introduced
  in \fref{3-thm1-8}--\fref{3-thm1-10}. Let
  \begin{align*}
    f = \fEl{f}{+} + \fEl{f}{-}   \qKomma
  \end{align*}
  where $\fEl{f}{+} :=  \fEl{P}{+}  f$. Then
  \begin{align*}
    f =  \fEl{P}{ \fEl{\dC}{\dT} }  f + h
      =  \fEl{P}{ \fEl{\dC}{\dT} }  f
          + \fEl{h}{\gG}  + \fEl{h}{\gF} \qPunkt
  \end{align*}
  This implies
  \begin{align*}
    \fEl{P}{+}  f
    =  \fEl{P}{ \fEl{\dC}{\dT} }  f
          + \fEl{h}{\gG}
          + \rklamFunk{ \fEl{h}{\gF} }{0} \cdot \dEins \qPunkt
  \end{align*}
  This means
  \begin{align}       \label{3-thm1-17}
    \fEl{P}{+}  f
    =  \fEl{P}{ \fEl{\dC}{\dT} }  f
          + \mkEl{ \cP^{\gF} }{\mu}{\gG}
            \fEl{P}{ \fEl{ \iEl{\gH}{n} }{\mu} }  f
          + \rklamFunk{ \fEl{h}{\gF} }{0} \cdot \dEins \qPunkt
  \end{align}
  The mapping
  $\fEl{h}{\gF} \mapsto \rklamFunk{ \fEl{h}{\gF} }{0}$
  is a linear functional on the set 
  $${\mathcal Pol}_{\le0}:=\Lin\{t^{-k}:k\in\N_0\}.$$
  Since $\mu$ is a Szeg\H{o} measure, the Szeg\H{o}-Kolmogorov-Krein
  Theorem (see, e.g. \cite[Theorem 4.31]{3}) implies that this functional
  is bounded in $\fEl{L^2}{\mu}$ on the set ${\mathcal Pol}_{\le0}$. Thus,
  there exists a constant $C \in \intervalO{0}{+\infty}$ such that
  \begin{align}       \label{3-thm1-18}
    \inorm{
        \rklamFunk{ \fEl{h}{\gF} }{0} \cdot \dEins
    }{ \fEl{L^2}{\mu} }
    &=    \abs{ \rklamFunk{ \fEl{h}{\gF} }{0} } \nonumber \\[4pt]
    &\leq C \cdot \norm{ \fEl{h}{\gF} } \nonumber \\[4pt]
    &= C
      \cdot 
      \norm{ 
        \rklam{ I - \mkEl{ \cP^{\gF} }{\mu}{\gG}  } 
        h
      }    \nonumber \\[4pt]
    &\leq C
      \cdot 
      \norm{  I - \mkEl{ \cP^{\gF} }{\mu}{\gG}  }
      \cdot
      \norm{h}    \nonumber \\[4pt]
    &\leq C
      \cdot 
      \norm{  I - \mkEl{ \cP^{\gF} }{\mu}{\gG}  }
      \cdot
      \norm{f}    \qPunkt
  \end{align}
  From \fref{3-thm1-17} and \fref{3-thm1-18} we get
  \begin{align*}
    \norm{ \fEl{P}{+} f }
    &\leq 
      \norm{ \fEl{P}{ \fEl{\dC}{\dT} } f }
      + \norm{
          \mkEl{ \cP^{\gF} }{\mu}{\gG}
            \fEl{P}{ \fEl{ \iEl{\gH}{n} }{\mu} }  f
        }
      + C \cdot \norm{ I - \mkEl{ \cP^{\gF} }{\mu}{\gG} }
          \cdot \norm{f} \nonumber \\[4pt]
    &\leq
      \norm{f} 
       +  \norm{ \mkEl{ \cP^{\gF} }{\mu}{\gG} } \cdot \norm{f} 
       + C \cdot \norm{ I - \mkEl{ \cP^{\gF} }{\mu}{\gG} }
          \cdot \norm{f}  \qPunkt
  \end{align*}
  Now considering the limit as $n \rightarrow \infty$
  and taking into account \fref{3-thm1-4}, we see that the
  boundedness of the projection $\mkEl{ \cP^{\gF} }{\mu}{\gG}$
  implies the boundedness of the Riesz projection $\fEl{P}{+}$
  in $\fEl{L^2}{\mu}$.
\end{proof}


\section{On the connection of the Riesz projection
          $\fEl{P}{+}$ with the orthogonal projections from
          $\fEl{\gH}{\mu}$ onto the subspaces
          $\mkEl{ \orth{\gH} }{\mu}{\gF}$
          and $\mkEl{ \orth{\gH} }{\mu}{\gG}$}
\label{s4}


Let $\mu \in \rklamFunk{ \fEl{\cM^{1}}{+} }{ \dT }$ be
a Szeg\H{o} measure. As we did earlier, we consider the simple
unitary colligation $\fEl{\Delta}{\mu}$ of type \fref{Nr.2.2} which is
associated with the measure $\mu$. As was previously mentioned, we
then have
\begin{align*}
  \mkEl{ \orth{\gH} }{\mu}{\gF} \neq  \gklam{0}
  \qquad
  \qquad
  \mbox{and}
  \qquad
  \qquad
  \mkEl{ \orth{\gH} }{\mu}{\gG} \neq  \gklam{0} \qPunkt
\end{align*}
We denote by $\fEl{P}{ \mkEl{ \orth{\gH} }{\mu}{\gF} }$
and $\fEl{P}{ \mkEl{ \orth{\gH} }{\mu}{\gG} }$ the
orthogonal projections from $\fEl{\gH}{\mu}$ onto
$\mkEl{ \orth{\gH} }{\mu}{\gF}$ and $\mkEl{ \orth{\gH} }{\mu}{\gG}$,
respectively.

Let $h \in \fEl{ \iEl{\gH}{n} }{\mu}$. Along with the decomposition
\fref{3-thm1-9} we consider the decomposition
\begin{align}       \label{4-1}
  h = \fEl{ \wt{h} }{\gF} + \fEl{ \orth{ \wt{h} } }{\gF}  \qKomma
\end{align}
where
\begin{align*}
 \wt h_\gF\in\gH_{\mu,\gF}^{(n)} 
  \qquad\qquad\mbox{and}\qquad\qquad
  \fEl{ \orth{ \wt{h} } }{\gF}
    \in \fEl{ \iEl{\gH}{n} }{\mu} 
    \ominus 
    \mkEl{ \iEl{\gH}{n} }{\mu}{\gF} \qPunkt 
\end{align*}
From the shape \fref{3-thm1-5} of the
subspace $\mkEl{ \iEl{\gH}{n} }{\mu}{\gF}$
and the polynomial structure of the orthonormal basis
$\fKlam{\varphi}{n}{0}{\infty}$ of the subspace
$\mkEl{\gH}{\mu}{\gF}$, it follows that
\begin{align*}
  \fEl{ \orth{ \wt{h} } }{\gF}
  = \fEl{P}{ \mkEl{ \orth{\gH} }{\mu}{\gF} }  h \qPunkt
\end{align*}
Since $\fEl{h}{\gF}$ (see \fref{3-thm1-9})
and $\fEl{ \wt{h} }{\gF}$ belong to $\mkEl{ \iEl{\gH}{n} }{\mu}{\gF}$,
we get
\begin{align}       \label{4-2}
  \fEl{ \orth{ \wt{h} } }{\gF}
  = \fEl{P}{ \mkEl{ \orth{\gH} }{\mu}{\gF} }  h
  = \fEl{P}{ \mkEl{ \orth{\gH} }{\mu}{\gF} }  \mkEl{h}{\gG}
  = \fEl{P}{ \mkEl{ \orth{\gH} }{\mu}{\gF} }
    \mkEl{ \cP^{\gF} }{\mu}{\gG} h
  = \mkEl{B}{\mu}{\gF}
    \mkEl{ \cP^{\gF} }{\mu}{\gG} h  \qKomma
\end{align}
where
\begin{align}       \label{4-3}
B_{\mu,\gF}:=\Rstr_{\gH_{\mu,\gG}}P_{\gH_{\mu,\gF}^\perp}:\gH_{\mu,\gG}\to\gH_{\mu,\gF}^\perp\,,
\end{align}
i.e., we consider $\mkEl{B}{\mu}{\gF}$ as an operator acting between
the spaces $\mkEl{\gH}{\mu}{\gG}$ and $\mkEl{ \orth{\gH} }{\mu}{\gF}$.

\begin{thm}       \label{4-thm1}
  Let $\mu \in \rklamFunk{ \fEl{\cM^{1}}{+} }{ \dT }$ be
  a Szeg\H{o} measure. Then the projection $\mkEl{ \cP^{\gF} }{\mu}{\gG}$
  is bounded in $\fEl{\gH}{\mu}$ if and only if the operator
  $\mkEl{B}{\mu}{\gF}$ defined in \fref{4-3} is boundedly invertible.
\end{thm}

\begin{proof}
  Suppose first that $\mkEl{B}{\mu}{\gF}$ has a bounded inverse
  $\mkEl{ \Inv{B} }{\mu}{\gF}$. Then for $h \in \fEl{ \iEl{\gH}{n} }{\mu}$,
  in view of \fref{4-1} and \fref{4-2}, we have
  \begin{align*}
    \mkEl{ \cP^{\gF} }{\mu}{\gG}  h
    = \mkEl{ \Inv{B} }{\mu}{\gF}
      \fEl{P}{ \mkEl{ \orth{\gH} }{\mu}{\gF} }  h \qPunkt
  \end{align*}
  If $n \rightarrow \infty$, this gives us the boundedness of
  the projection $\mkEl{ \cP^{\gF} }{\mu}{\gG}$ in $\fEl{\gH}{\mu}$.
  
  Conversely, suppose that the projection $\mkEl{ \cP^{\gF} }{\mu}{\gG}$
  is bounded in $\fEl{\gH}{\mu}$.
  If $h \in \fEl{ \iEl{\gH}{n} }{\mu} \ominus \mkEl{ \iEl{\gH}{n} }{\mu}{\gF}$,
  then the decomposition \fref{4-1} provides us
  \begin{align*}
    h = \fEl{ \wt{h} }{\gF}^\bot
  \end{align*}
  and identity \fref{4-2} yields
  \begin{align}       \label{4-4}
    h = \mkEl{ B }{\mu}{\gF}
        \mkEl{ \cP^{\gF} }{\mu}{\gG}  h \qPunkt
  \end{align}
  Since $\mkEl{ \cP^{\gF} }{\mu}{\gG}$ is bounded in $\fEl{\gH}{\mu}$,
  Theorem \ref{3-thm1} implies that the Riesz projection $\fEl{P}{+}$
  is bounded in $\fEl{L^2}{\mu}$. Then it follows from condition (iii)
  in Theorem \ref{s0-t1} that
  \begin{align*}
    \mkEl{\gH}{\mu}{\gF}  \cap  \mkEl{\gH}{\mu}{\gG}  = \gklam{0} \qPunkt
  \end{align*}
  Thus, from the shape \fref{4-3} of the operator $\mkEl{B}{\mu}{\gF}$,
  we infer that
  \begin{align*}
    \mathrm{ker} \, \mkEl{B}{\mu}{\gF} = \gklam{0}
  \qquad
  \qquad
  \mbox{and}
  \qquad
  \qquad
    \rklamFunk{ \mkEl{ \cP^{\gF} }{\mu}{\gG} }
              { \fEl{ \orth{\gH} }{\gF} }
      = \fEl{ \gH }{\gG}  \qPunkt
  \end{align*}
  Now equation \fref{4-4} can be rewritten in the form
  \begin{align*}
    \mkEl{ \Inv{B} }{\mu}{\gF}  h 
    = \mkEl{ \cP^{\gF} }{\mu}{\gG}  h \qKomma
  \qquad
  \qquad
  \qquad
  \qquad
  h
    \in \fEl{ \iEl{\gH}{n} }{\mu} 
    \ominus 
    \mkEl{ \iEl{\gH}{n} }{\mu}{\gF} \qPunkt 
  \end{align*}
  The limit $n \rightarrow \infty$, \eqref{3-thm1-2} and \eqref{3-thm1-4}  give us the desired result.
\end{proof}

The combination of Thereom \ref{4-thm1} with Theorem \ref{3-thm1}
leads us to the following result.

\begin{thm}       \label{4-thm2}
  Let $\mu \in \rklamFunk{ \fEl{\cM^{1}}{+} }{ \dT }$ be
  a Szeg\H{o} measure. Then the Riesz projection $\fEl{P}{+}$
  is bounded in $\fEl{L^2}{\mu}$ if and only if the operator
  $\mkEl{B}{\mu}{\gF}$ defined in \fref{4-3} is boundedly invertible.
\end{thm}

  Let $f \in {\mathcal Pol}$ be given by \fref{s0-1}.
  Along with the Riesz projection $\fEl{P}{+}$, we consider
  the projection $\fEl{P}{-}$, which is defined by:
  \begin{align*}
    \rklamFunk{ \rklam{ \fEl{P}{-} } }{t}
    :=  \sum_{ -k \in I \cap \N_0 }
          { \fEl{a}{k}  t^k } \qKomma
        \qquad  \qquad  t \in \dT   \qPunkt
  \end{align*}
  Obviously,
  \begin{align*}
    \fEl{P}{-}  = \overline{
                    \fEl{P}{+}  \overline{f}
                  } \qKomma
  \qquad
  \qquad
  \mbox{and}
  \qquad
  \qquad
    \fEl{P}{+}  = \overline{
                    \fEl{P}{-}  \overline{f}
                  } \qPunkt
  \end{align*}
Thus, the boundedness of one of the projections $\fEl{P}{+}$
and $\fEl{P}{-}$ in $\fEl{L^2}{\mu}$ implies the boundedness
of the other one. It is readily checked that the change from the
projection $\fEl{P}{+}$ to $\fEl{P}{-}$ is connected
with changing the roles of the spaces $\mkEl{\gH}{\mu}{\gG}$
and $\mkEl{\gH}{\mu}{\gF}$. Thus we obtain the following result,
which is dual to Theorem \ref{4-thm2}.

\begin{thm}       \label{4-thm3}
  Let $\mu \in \rklamFunk{ \fEl{\cM^{1}}{+} }{ \dT }$ be
  a Szeg\H{o} measure. Then the Riesz projection $\fEl{P}{+}$
  is bounded in $\fEl{L^2}{\mu}$ if and only if the operator
  $\mkEl{B}{\mu}{\gG} : \mkEl{\gH}{\mu}{\gF} \rightarrow \mkEl{ \orth{\gH} }{\mu}{\gG}$ defined by
  \begin{align}\label{4-thm3-1}
    \mkEl{B}{\mu}{\gG}h
    := {\fEl{P}{ \mkEl{ \orth{\gH} }{\mu}{\gG} }}h
  \end{align}
  is boundedly invertible. Here the symbol $\fEl{P}{ \mkEl{ \orth{\gH} }{\mu}{\gG} }$
  stand for the orthogonal projection from $\fEl{\gH}{\mu}$ onto
  $\mkEl{ \orth{\gH} }{\mu}{\gG}$.
\end{thm}


\section{Matrix Representation of the Operator $\mkEl{B}{\mu}{\gG}$
          in Terms of the Schur Parameters Associated With the
          Measure $\mu$}
\label{s5}


Let $\mu \in \rklamFunk{ \fEl{\cM^{1}}{+} }{ \dT }$ be
a Szeg\H{o} measure. We consider the simple unitary
colligation $\fEl{\Delta}{\mu}$ of the type \fref{Nr.2.2} which
is associated with the measure $\mu$. In this case we have
(see Section \ref{s2})
\begin{align*}
  \mkEl{ \orth{\gH} }{\mu}{\gF} \neq  \gklam{0}
  \qquad
  \qquad
  \mbox{and}
  \qquad
  \qquad
  \mkEl{ \orth{\gH} }{\mu}{\gG} \neq  \gklam{0}
\end{align*}
The operator $\mkEl{B}{\mu}{\gG}$ acts between the
subspaces $\mkEl{\gH}{\mu}{\gF}$ and $\mkEl{ \orth{\gH} }{\mu}{\gG}$.
According to the matrix description of the operator
$\mkEl{B}{\mu}{\gG}$ we consider particular orthogonal
bases in these subspaces. In the subspace $\mkEl{ \gH }{\mu}{\gF}$ we have already
considered one such basis, namely the basis consisting
of the trigonometric polynomials $\fKlam{\varphi}{n}{1}{\infty}$
(see Theorem \ref{s2-t2}). Regarding the construction of
an orthonormal basis in $\mkEl{ \orth{\gH} }{\mu}{\gG}$,
we first complete the system $\fKlam{\varphi}{n}{1}{\infty}$
to an orthonormal basis in $\fEl{\gH}{\mu}$. This procedure
is described in more detail in \cite{7}.

We consider the orthogonal decomposition
\begin{align}\label{Nr.67}
& \gH_\mu = \gH_{\mu,\gF} \oplus \gH_{\mu,\gF}^\perp.
\end{align}
Denote by $\tilde{\gL}_0$ the wandering subspace which generates the subspace associated with the unilateral shift 
$V_{T_\mu^\ast}$. Then (see Proposition \ref{3.4})
$\dim \tilde{\gL}_0 = 1$ and, since $V_{T_\mu^\ast}$ is an isometric operator, we have
\begin{align}\label{Nr.3.10}
& V_{T_\mu^\ast} = {\rm Rstr.}_{\gH_{\mu,\gF}^\perp} (U_\mu^\times)^\ast.
\end{align}
Consequently,
\begin{align}\label{Nr.3.11}
& \gH_{\mu,\gF}^\perp 
 = \bigoplus\limits_{n=0}^\infty V_{T_\mu^\ast}^n (\tilde{\gL}_0)
 = \bigvee_{n=0}^\infty (T_\mu^\ast)^n (\tilde{\gL}_0)
 = \bigvee_{n=0}^\infty [ (U_\mu^\times)^\ast]^n (\tilde{\gL}_0).
\end{align}
There exists (see \cite[Corollary~1.10]{7}) a unique unit function $\psi_1\in\tilde{\gL}_0$ which fulfills
\begin{align}\label{Nr.3.12}
& \bigl( G_{\mu}^\ast (1), \psi_1 \bigr)_{L_\mu^2} > 0.
\end{align}
Because of (\ref{Nr.3.10}), (\ref{Nr.3.11}), and (\ref{Nr.3.12}) it follows that the
sequence $(\psi_k)_{k=1}^\infty$, where
\begin{align}\label{Nr.3.13}
& \psi_k := [(U_\mu^\times)^\ast]^{k-1} \psi_1,
 \quad k\in\N,
\end{align}
is the unique orthonormal basis of the space $\gH_{\mu,\gF}^\perp$ which satisfies the conditions
\begin{align}\label{Nr.3.14}
& \bigl( G_{\mu}^\ast (1), \psi_1 \bigr)_{L_\mu^2} > 0,
 \quad \psi_{k+1} = (U_\mu^\times)^\ast \psi_k,
 \quad k\in \N,
\end{align}
or equivalently
\begin{align}\label{Nr.3.15}
 & \bigl( G_{\mu}^\ast (1), \psi_1 \bigr)_{L_\mu^2} > 0,
 \quad \psi_{k+1} (t) = t^k\cdot \psi_1 (t),
 \; t\in\mathbb T, \quad k\in \N.
\end{align}
According to the considerations in \cite{7} we introduce the following notion.

\begin{defn}\label{3.5}
 The constructed orthonormal basis
\begin{align}\label{Nr.3.16}
 & \varphi_0, \varphi_1, \varphi_2, \ldots ;\;  \psi_1, \psi_2, \ldots
\end{align}
 in the space $L_\mu^2$ which satisfies the conditions \eqref{Nr.3.2} and \eqref{Nr.3.14} is called the
 \emph{cano\-ni\-cal orthonormal basis in} $L_\mu^2$.
\end{defn}

Note that the analytic structure of the system $(\psi_k)_{k=1}^\infty$ is described in the paper \cite{10}.

Obviously, the canonical orthonormal basis \eqref{Nr.3.16} in $L_\mu^2$ is uniquely determined
by the conditions \eqref{Nr.3.2} and \eqref{Nr.3.14}. 
Here the sequence $(\varphi_k)_{k=0}^\infty$ is an orthonormal system of
polynomials (depending on $t^*$). The orthonormal system $(\psi_k)_{k=1}^\infty$ is
built with the aid of the operator $U_\mu^\times$ from the function
$\psi_1$ $($see \eqref{Nr.3.13}$)$ in a similar way as the system
$(\varphi_k)_{k=0}^\infty$ was built from the function $\varphi_0$
$($see \eqref{Nr.3.1} and \eqref{Nr.3.2}$)$. The only difference is that the system
$\left( [(U_\mu^\times)^\ast ]^k \psi_1 \right)_{k=0}^\infty$ is
orthonormal, whereas in the general case the system $\left(
(U_\mu^\times)^k \varphi_0\right)_{k=0}^\infty$ is not
orthonormal. In this respect, the sequence $(\psi_k)_{k=1}^\infty$
can be considered as a natural completion of the system of
orthonormal polynomials $(\varphi_k)_{k=0}^\infty$ to an orthonormal
basis in $L_\mu^2$.


\begin{rem}\label{3.6}
The orthonormal system
\begin{align}\label{Nr.3.16b}
& \varphi_1,\varphi_2,\ldots ;\;  \psi_1, \psi_2,\ldots
\end{align}
is an orthonormal basis in the space $\gH_\mu$. We will call it
the \emph{canonical orthonormal basis in} $\gH_\mu$. 
\end{rem}

It is well known (see, e.g., Brodskii \cite{8}) that one can consider simultaneously together with the simple unitary
colligation (\ref{Nr.2.2}) the adjoint unitary colligation
\begin{equation}\label{Nr.3.29}
 \tilde{\bigtriangleup}_\mu := 
 (\gH_\mu, \mathbb C, \mathbb C;T_\mu^\ast, G_\mu^\ast, F_\mu^\ast, S_\mu^\ast)
\end{equation}
which is also simple. Its characteristic function $\Theta_{\tilde{\bigtriangleup}_\mu}$ is for
each $z\in\mathbb D$ given by
\begin{equation*}
 \Theta_{\tilde{\bigtriangleup}_\mu} (z) = \Theta_{\bigtriangleup_\mu}^{\ast} (z^{\ast}).
\end{equation*}
We note that the unitary colligation (\ref{Nr.3.29}) is associated with the operator $(U_\mu^\times)^\ast$.
It can be easily checked that the action of $(U_\mu^\times)^\ast$ is given for each
$ f\in L_\mu^2$ by
\[ [(U_\mu^\times)^\ast f] (t) = t\cdot f(t),\quad t\in\mathbb T . \]
If we replace the operator $U_\mu^\times $ by $(U_\mu^\times)^\ast$ in the preceding considerations,
which have lead to the canonical orthonormal basis (\ref{Nr.3.16}), then we obtain an orthonormal basis
of the space $L_\mu^2$ which consists of two sequences
\begin{equation}
  \label{Nr.ONB}
  (\tilde{\varphi}_j)_{j=0}^\infty 
\quad\mbox{and}\quad
  (\tilde{\psi}_j)_{j=1}^\infty
\end{equation}
of functions. From our treatments above it follows that the orthonormal basis (\ref{Nr.ONB}) is uniquely
determined by the following conditions:
\begin{enumerate}
\item[(a)]
The sequence $(\tilde{\varphi}_j)_{k=0}^\infty$ arises from the result of the Gram-Schmidt orthogonalization
procedure of the sequence $\left( [(U_\mu^\times)^\ast]^n {\bf 1}\right)_{n=0}^\infty$ and additionally
taking into account the normalization conditions
\begin{equation*}
 \bigl( [(U_\mu^\times)^\ast]^n {\bf 1}, \tilde{\varphi}_n \bigr)_{L_\mu^2} > 0,
\quad n\in \N_0.
\end{equation*}
\item[(b)]
The relations
\[ \bigl(F_{\mu}(1), \tilde{\psi}_1\bigr)_{L_\mu^2} > 0
\quad\mbox{and}\quad
   \tilde{\psi}_{k+1} = U_\mu^\times \tilde{\psi}_k,\quad k\in \N, \]
hold. 
\end{enumerate}
It can be easily checked that
\[ \tilde{\varphi}_{k} = \varphi_k^{\ast}, 
\quad k\in \N_0, 
\vspace{-2mm} \]
and \vspace{-2mm}
\[ \tilde{\psi}_{k} = \psi_k^{\ast},
\quad k\in \N.
\]

According to the paper \cite{7} we introduce the following notion.

\begin{defn}\label{3.14}
The orthogonal basis
\begin{equation}\label{Nr.3.31}
 \varphi_0^{\ast}, \varphi_1^{\ast}, \varphi_2^{\ast},\ldots ; \psi_1^{\ast}, \psi_2^{\ast}, \ldots
\end{equation}
is called the conjugate canonical orthonormal basis with respect to the canonical orthonormal basis \eqref{Nr.3.16}.
\end{defn}

We note that $\varphi_0=\varphi_0^*=1$.
Similarly as (\ref{Nr.3.8}) the identity
\begin{equation}\label{Nr.3.38}
 \bigvee_{k=0}^{n} [(U_\mu^\times)^\ast]^k {\bf 1} 
 = \left( \bigvee_{k=0}^{n-1} (T_\mu^\ast)^k G_\mu^\ast (1) \right) \oplus \mathbb C_\mathbb T
\end{equation}
can be verified. Thus,
 \begin{equation}\label{Nr.3.39}
\gH_{\mu,\gF} = \bigvee_{k=1}^\infty \varphi_k,\quad \gH_{\mu,\gG}
= \bigvee_{k=1}^\infty \varphi_k^{\ast},
\end{equation}
 \begin{equation}\label{Nr.3.40}
\gH_{\mu,\gF}^\perp = \bigvee_{k=1}^\infty \psi_k,\quad
\gH_{\mu,\gG}^\perp = \bigvee_{k=1}^\infty \psi_k^{\ast}.
\end{equation}

In \cite[Chapter 3]{7} the unitary operator $\cU$ was introduced  which
maps the elements of the canonical basis (\ref{Nr.3.16}) onto the corresponding elements of the
conjugate canonical basis (\ref{Nr.3.31}). More precisely, we consider the operator
\begin{equation}\label{5-16}
 \cU_\mu \varphi_n = \varphi_n^{\ast},\quad n\in \N_0,
\qquad \mbox{and} \qquad 
 \cU_\mu \psi_n = \psi_n^{\ast},\quad n\in \N.
\end{equation}
The operator $\cU_\mu$ is related to the conjugation operator in $L_\mu^2$. Namely, if
$f\in L_\mu^2$ and if
 \begin{equation*}
f =\sum_{k=0}^\infty \alpha_k \varphi_k + \sum_{k=1}^\infty
\beta_k\psi_k,
\end{equation*}
then
\begin{equation*}
 f^{\ast} = \sum_{k=0}^\infty \alpha_k^{\ast} \varphi_k^{\ast} + \sum_{k=1}^\infty \beta_k^{\ast} \psi_k^{\ast} 
          = \sum_{k=0}^\infty \alpha_k^{\ast} \cU\varphi_k + \sum_{k=1}^\infty \beta_k^{\ast} \cU\psi_k.
\end{equation*}

From \fref{5-16} it follows that
\begin{align*}
  \fEl{\cU}{\mu}  ~:~ \fEl{\gH}{\mu}  \longrightarrow \fEl{\gH}{\mu}  \;, \qquad \cU_\mu(\bf 1)=\bf 1 \,.
\end{align*}
Let
\begin{align}       \label{5-17}
  \fEl{\cU}{ \fEl{\gH}{\mu} } 
    :=   \Rstr_{ \fEl{\gH}{\mu} }{\fEl{\cU}{\mu}}  \qPunkt
\end{align}
Then, obviously,
\begin{align}       \label{5-18}
  \fEl{\cU}{ \fEl{\gH}{\mu} } \fEl{\varphi}{n}  = \Adj{ \fEl{\varphi}{n} }
  \qquad
  \mbox{and}
  \qquad
  \fEl{\cU}{ \fEl{\gH}{\mu} } \fEl{\psi}{n}  = \Adj{ \fEl{\psi}{n} }
  \qKomma
  \qquad
  n \in \N \qPunkt
\end{align}
Clearly, the system $\fKlam{ \Adj{\psi} }{n}{1}{\infty}$
is an orthonormal basis in the space $\mkEl{ \orth{\gH} }{\mu}{\gG}$.
This system will turn out to be the special orthonormal basis of
the space $\mkEl{ \orth{\gH} }{\mu}{\gG}$ mentioned
at the beginning of this section. Thus, the matrix representation
of the operator
\begin{align*}
  \mkEl{B}{\mu}{\gG}  
  ~:~ \mkEl{\gH}{\mu}{\gF}  \longrightarrow \mkEl{ \orth{\gH} }{\mu}{\gG}
\end{align*}
will be considered with respect to the orthonormal bases
\begin{align}       \label{5-19}
  \fKlam{\varphi}{n}{1}{\infty}
  \qquad
  \qquad
  \mbox{and}
  \qquad
  \qquad
  (\psi_n^*)_{n=1}^\infty
\end{align}
of the spaces $\mkEl{\gH}{\mu}{\gF}$ and $\mkEl{ \orth{\gH} }{\mu}{\gG}$,
respectively. Let
\begin{align}       \label{5-20}
  \begin{pmatrix}
    \cR   & \cL   \\[5pt]
    \cP   & \cQ
  \end{pmatrix}
\end{align}
be the matrix representation of the operator $\fEl{\cU}{ \fEl{\gH}{\mu} }$
with respect to the canonical basis \fref{Nr.3.16b} of the space $\fEl{\gH}{\mu}$.
Then, from \fref{5-18} we infer that the columns
\begin{align*}
  \begin{pmatrix}
    \cR   \\[5pt]
    \cP
  \end{pmatrix}
  \qquad
  \qquad
  \mbox{and}
  \qquad
  \qquad
  \begin{pmatrix}
    \cL   \\[5pt]
    \cQ
  \end{pmatrix}
\end{align*}
of the block-matrix \fref{5-20} are the coefficients in the series
developments of $\Adj{ \fEl{\varphi}{n} }$ and $\Adj{ \fEl{\psi}{n} }$
with respect to the canonical basis \fref{Nr.3.16b}. If $h \in \fEl{\gH}{\mu}$
then clearly
\begin{align}       \label{5-21}
  \fEl{P}{ \mkEl{ \orth{\gH} }{\mu}{\gG} }  h
  = \sum_{k = 1}^{\infty}
    { 
      \rklamPaar{h}{ \Adj{ \fEl{\psi}{k} } } 
      \Adj{ \fEl{\psi}{k} }
    } \qPunkt
\end{align}
Thus, the matrix representation of the operator
$\fEl{P}{ \mkEl{ \orth{\gH} }{\mu}{\gG} }$ considered as
an operator acting between $\fEl{\gH}{\mu}$ and
$\mkEl{ \orth{\gH} }{\mu}{\gG}$ equipped with the
orthonormal bases \fref{Nr.3.16b} and $\fKlam{ \Adj{\psi} }{n}{1}{\infty}$
has the form
\begin{align*}
  \rklamPaar{ \Adj{\cL} }{ \Adj{\cQ} }  \qPunkt
\end{align*}
From this and the shape \fref{4-thm3-1} of the operator
$\mkEl{B}{\mu}{\gG}$, we obtain the following result.

\begin{thm}       \label{5-thm4}
  Let $\mu \in \rklamFunk{ \fEl{\cM^{1}}{+} }{ \dT }$ be
  a Szeg\H{o} measure. Then the matrix of the operator
  \begin{align*}
    \mkEl{B}{\mu}{\gG}  
     ~:~ \mkEl{\gH}{\mu}{\gF}  \longrightarrow \mkEl{ \orth{\gH} }{\mu}{\gG}
  \end{align*}
  with respect to the orthonormal bases $\fKlam{\varphi}{k}{1}{\infty}$
  and $\fKlam{ \Adj{\psi} }{n}{1}{\infty}$ of the spaces
  $\mkEl{\gH}{\mu}{\gF}$ and $\mkEl{ \orth{\gH} }{\mu}{\gG}$,
  respectively, is given by $\Adj{\cL}$ where $\cL$ is the block of
  the matrix given in \fref{5-20}. 
\end{thm}

Now Theorem \ref{4-thm3} can be reformulated in the following way.

\begin{cor}       \label{5-cor5}
  Let $\mu \in \rklamFunk{ \fEl{\cM^{1}}{+} }{ \dT }$ be
  a Szeg\H{o} measure. Then the Riesz projection $\fEl{P}{+}$
  is bounded in $\fEl{L^2}{\mu}$ if and only if 
  $\cL^*$ is boundedly invertible in $l_2$ where $\cL$ is the block of the matrix given in 
  \fref{5-20}.
\end{cor}

In \cite[Corollary 3.7]{7} the matrix $\cL$ was expressed
in terms of the Schur parameters associated with the measure
$\mu$. In order to write down this matrix we introduce the
necessary notions and terminology used in \cite{7}. The matrix
$\cL$ expressed in terms of the corresponding Schur parameter
sequence will the denoted by $\rklamFunk{\cL}{\gamma}$.

Let$$
  \Gamma l_2
  :=\left\{\gamma=(\gamma_j)_{j=0}^\infty\in l_2:\gamma_j\in\D,j\in\N_0\right\}.
$$
Thus, $\Gamma l_2$ is the subset of all $\gamma=(\gamma_j)_{j=0}^\infty\in\Gamma$, for which the product
$$\prod_{j=0}^\infty\left(1-|\gamma_j|^2\right)$$
converges. 

Let us mention the following well--known fact (see, for example,
Remark \ref{s0-r1})

\begin{prop}       \label{5-prop5}
  Let $\mu \in \rklamFunk{ \fEl{\cM^{1}}{+} }{ \dT }$. Then
  $\mu$ is a Szeg\H{o} measure if and only if
  $ \gamma$ belongs to $\Gamma \fEl{l}{2} $.
\end{prop}

For a Schur parameter sequence $\gamma$ belonging to $\Gamma l_2$, we note that the sequence $(L_n(\gamma))_{n=0}^\infty$ introduced in formula (3.12) of \cite{7} via
\begin{eqnarray}\label{LnGamma}
&&L_0(\gamma):=1\textnormal{ and, for each positive integer }n,\textnormal{ via } L_n(\gamma):=\nonumber\\
&&\sum_{r=1}^n(-1)^r\sum_{s_1+s_2+\ldots+s_r=n}\sum_{j_1=n-s_1}^\infty\sum_{j_2=j_1-s_2}^\infty\ldots\sum_{j_r=j_{r-1}-s_r}^\infty\gamma_{j_1}\overline\gamma_{j_1+s_1}\ldots\gamma_{j_r}\overline\gamma_{j_r+s_r}\nonumber\\
\end{eqnarray}
plays a key role. Here the summation runs over all ordered $r$--tuples $(s_1,\ldots,s_r)$ of positive integers which satisfy $s_1+\ldots+s_r=n$. For example,
$$L_1(\gamma)=-\sum_{j=0}^\infty \gamma_j\overline{\gamma_{j+1}}$$
and
$$L_2(\gamma)=-\sum_{j=0}^\infty \gamma_j\overline{\gamma_{j+2}}+\sum_{j_1=1}^\infty\sum_{j_2=j_1-1}^\infty\gamma_{j_1}\overline{\gamma_{j_1+1}}\gamma_{j_2}\overline{\gamma_{j_2+1}}.$$
Obviously, if $\gamma\in\Gamma l_2$, then the series (\ref{LnGamma}) converges absolutely.

For each $\gamma=(\gamma_j)_{j=0}^\infty\in\Gamma l_2$, we set
\begin{equation}\label{Pik}
\Pi_k:=\prod_{j=k}^\infty D_{\gamma_j},\ \ k\in\N_0,
\end{equation}
where
\begin{equation}\label{DGamma}
D_{\gamma_j}:=\sqrt{1-|\gamma_j|^2},\ \ j\in\N_0.
\end{equation}
In the space $l_2$ we define the coshift mapping $W:l_2\rightarrow l_2$ via
\begin{equation}\label{coshift}
 (z_j)_{j=0}^\infty\mapsto(z_{j+1})_{j=0}^\infty.
\end{equation}

The following result is contained in \cite[Theorem 3.6, Corollary 3.7]{7}.

\begin{thm}       \label{5-thm7}
  Let $\mu \in \rklamFunk{ \fEl{\cM^{1}}{+} }{ \dT }$ be
  a Szeg\H{o} measure and let $\gamma \in \Gamma$ be the Schur
  parameter sequence associated with $\mu$. Then
  $ \gamma  \in \Gamma \fEl{l}{2} $
  and the block $\cL$ of the matrix \fref{5-20} has the form
  \begin{align}\label{5-thm7-1}
   &  \rklamFunk{\cL}{\gamma}
     =  \begin{pmatrix}        
         \Pi_1 & 0 & 0 &  \ldots \\
         \Pi_2L_1(W\gamma) & \Pi_2 & 0 & \ldots  \\
       \Pi_3L_2(W\gamma) & \Pi_3L_1(W^2\gamma) & \Pi_3 & \ldots  \\
          \vdots & \vdots & \vdots & \ddots \\
\Pi_nL_{n-1}(W\gamma) & \Pi_nL_{n-2}(W^2\gamma) & \Pi_nL_{n-3}(W^3\gamma) & \ldots \\
    \vdots & \vdots & \vdots &          
                \end{pmatrix}    \qKomma
  \end{align}
  where $\fEl{\Pi}{j}$, $\rklamFunk{ \fEl{L}{j} }{\gamma}$
  and $W$ are given via the formulas \fref{Pik}, \fref{LnGamma},
  and \fref{coshift}, respectively.
\end{thm}

\begin{rem}       \label{5-rem8}
  It follows from Theorems \ref{5-thm4} and \ref{5-thm7}
  that the matrix representation of the operator
  \begin{align*}
    \mkEl{B}{\mu}{\gG}  
     ~:~ \mkEl{\gH}{\mu}{\gF}  \longrightarrow \mkEl{ \orth{\gH} }{\mu}{\gG}
  \end{align*}
  with respect to the orthonormal bases $\fKlam{\varphi}{k}{1}{\infty}$
  and $\fKlam{\psi^*}{n}{1}{\infty}$ of the spaces $\mkEl{\gH}{\mu}{\gF}$
  and $\mkEl{ \orth{\gH} }{\mu}{\gG}$, respectively, is given by the
  matrix $\rklamFunk{ \Adj{\cL} }{\gamma}$, where $\rklamFunk{\cL}{\gamma}$
  has the form \fref{5-thm7-1}.
\end{rem}


\section{Characterization of Helson-Szeg\H o measures in terms of the Schur parameters of the associated Schur function}
\label{s6}

The first criterion which characterizes Helson-Szeg\H o measures in the associated Schur parameter sequence was already 
obtained. It follows by combination of Theorem \ref{s0-t1}, Theorem \ref{4-thm3}, Proposition \ref{5-prop5},
Theorem \ref{5-thm7}, and Remark \ref{5-rem8}. This leads us to the following theorem, which is one of the main results of this paper.

\begin{thm}\label{s6-t1}
Let $\mu\in\cM_+^1(\T)$ and let $\gamma\in\Gamma$ be the sequence of Schur parameters associated with $\mu$.
Then $\mu$ is a Helson-Szeg\H o measure if and only if $\gamma\in\Gamma l_2$ and the operator $\cL^*(\gamma)$,
which is defined in $l_2$ by the matrix \eqref{5-thm7-1}, is boundedly invertible.
\end{thm}

\begin{cor}\label{s6-c1}
Let $\mu\in\cM_+^1(\T)$ and let $\gamma\in\Gamma$ be the sequence of Schur parameters associated with $\mu$.
Then $\mu$ is a Helson-Szeg\H o measure if and only if $\gamma\in\Gamma l_2$ and there exists some positive constant $C$
such that for each $h\in l_2$ the inequality
\begin{align}\label{s6-c1-1}
 \|\cL^*(\gamma)h\|\ge C\|h\|
\end{align}
is satisfied.
\end{cor}

\begin{proof}
First suppose that $\gamma\in\Gamma l_2$ and that there exists some positive constant $C$ such that for each $h\in l_2$ the inequality \eqref{s6-c1-1} is satisfied. From the shape \eqref{5-thm7-1} of the operator $\cL(\gamma)$ it follows 
immediately that 
$\ker\cL(\gamma)=\{0\}$. Thus, $\overline{\Ran\cL^*(\gamma)}=l_2$. From \eqref{s6-c1-1} it follows that
the operator $\cL^*(\gamma)$ is invertible and that the corresponding inverse operator $\big(\cL^*(\gamma)\big)^{-1}$ is bounded and satisfies
$$ \big\|\big(\cL^*(\gamma)\big)^{-1}\big\|\le\frac1C $$
where $C$ is taken from \eqref{s6-c1-1}. Since $\cL^*(\gamma)$ is a bounded linear operator, the operator $[\cL^*(\gamma)]^{-1}$ is closed. Thus $\Ran\cL^*(\gamma)=l_2$ and, consequently, the operator $\cL^*(\gamma)$ is boundedly invertible. 
Hence, Theorem \ref{s6-t1} yields that $\mu$ is a Helson-Szeg\H o measure. 
If $\mu$ is a Helson-Szeg\H o measure, then Theorem \ref{s6-t1} yields that $\cL^*(\gamma)$ is boundedly invertible.
Hence, condition \eqref{s6-c1-1} is trivially satisfied.
\end{proof}

In order to derive criteria in another way we need some statements on the operator $\cL(\gamma)$ which were obtained in \cite{7}.

The following result which originates from \cite[Theorem 3.12 and Corollary 3.13]{7} plays an important role in the study of the matrix $\cL(\gamma)$. Namely, it describes the multiplicative structure of $\cL(\gamma)$ and indicates connections to the backward shift.

\begin{thm}       \label{s6-2A}
  It holds that
  \begin{equation}       \label{s6-2A-1A}
    \cL(\gamma) = \gM(\gamma) \cdot \cL(W \gamma)
  \end{equation}
  where
  \begin{equation}       \label{s6-2A-1B}
    \gM(\gamma) :=
    \rklam{
    \begin{smallmatrix}
      D_{\gamma_1} & 0 & 0 & \cdots & 0 & \cdots\\
      -\gamma_1\overline\gamma_2 & D_{\gamma_2} & 0 & \cdots & 0 & \cdots \\
      -\gamma_1D_{\gamma_2}\overline\gamma_3 & -\gamma_2\overline\gamma_3 & D_{\gamma_3} & \ldots & 0 & \cdots \\
                \vdots & \vdots & \vdots & & \vdots \\
      -\gamma_1\left( \prod_{j=2}^{n-1}D_{\gamma_j}\right)\overline\gamma_n & -\gamma_2\left( \prod_{j=3}^{n-1}D_{\gamma_j}\right)\overline\gamma_n & -\gamma_3\left( \prod_{j=4}^{n-1}D_{\gamma_j}\right)\overline\gamma_n & \cdots & D_{\gamma_n} & \cdots \\
      \vdots & \vdots & \vdots &  & \vdots & 
    \end{smallmatrix}
    }
  \end{equation}
  and $D_{\gamma_j} := \sqrt{ 1 - \lvert \gamma_j \rvert^2 }, \, j \in \N_0$.
  The matrix $\gM(\gamma)$ satisfies
  \begin{equation}       \label{s6-2A-1C}
    I - \gM(\gamma) \gM^*(\gamma) = \eta(\gamma) \eta^*(\gamma)
  \end{equation}
  where
  \begin{equation}       \label{s6-2A-1D}
    \eta(\gamma)
    := {\rm col} \left(
        \overline{ \gamma_1 } ,  \, 
        \overline{ \gamma_2 } D_{\gamma_1} ,  \, 
        \ldots ,
        \overline{ \gamma_n } \prod_{j = 1}^{n - 1}{ D_{\gamma_j} } ,  \, 
        \ldots
       \right)
  \end{equation}

\end{thm}


Let $\gamma\in\Gamma l_2$. For each $n\in\N$ we set (see formula (5.3) in \cite{7})
\begin{equation}\label{frakLn}
\fL_n(\gamma):=\begin{pmatrix}
 \Pi_1 & 0 & 0 & \ldots & 0 \\
 \Pi_2L_1(W\gamma) & \Pi_2 & 0 & \ldots & 0 \\
\Pi_3L_2(W\gamma) & \Pi_3L_1(W^2\gamma) & \Pi_3 & \ldots & 0 \\
          \vdots & \vdots & \vdots & & \vdots \\
\Pi_nL_{n-1}(W\gamma) & \Pi_nL_{n-2}(W^2\gamma) & \Pi_nL_{n-3}(W^3\gamma) & \ldots & \Pi_n
\end{pmatrix}.
\end{equation}
The matrices introduced in (\ref{frakLn}) will play an important role in our investigations. Now we turn our attention to some properties of the matrices $\fL_n(\gamma)$, $n\in\N $, which will later be of use. From Corollary 5.2 in \cite{7} we get the following result.

\begin{lem}\label{lem11a}
Let $\gamma=(\gamma_j)_{j=0}^\infty\in\Gamma l_2$ and let $n\in\N$. Then the matrix $\fL_n(\gamma)$ defined by (\ref{frakLn}) is contractive.
\end{lem}

We continue with some asymptotical considerations.

\begin{lem}\label{lem12neu}
Let $\gamma=(\gamma_j)_{j=0}^\infty\in\Gamma l_2$. Then:
\begin{itemize}
 \item[(a)] $\lim_{k\rightarrow\infty}\Pi_k=1$.
\item[(b)] For each $j\in\N$,\: $\lim_{m\rightarrow\infty}L_j(W^m\gamma)=0$.
\item[(c)] For each $n\in\N$,\: $\lim_{m\rightarrow\infty}\fL_n(W^m\gamma)=I_n$.
\end{itemize}
\end{lem}

\begin{proof}
The choice of $\gamma$ implies the convergence of the infinite product $\prod_{k=0}^\infty D_{\gamma_k}$. This yields (a). Assertion (b) is an immediate consequence of the definition of the sequence $(L_j(W^m\gamma))_{m=1}^\infty$ (see (\ref{LnGamma}) and (\ref{coshift})). By inspection of the sequence\\
$(\fL_n(W^m\gamma))_{m=1}^\infty$ one can immediately see that the combination of (a) and (b) yields the assertion of (c).
\end{proof}


The following result is given in \cite[Lemma 5.3]{7}.

\begin{lem}\label{lem13}
Let $\gamma=(\gamma_j)_{j=0}^\infty\in\Gamma l_2$ and let $n\in\N$. Then
\begin{equation}\label{LnProd}
\fL_n(\gamma)=\gM_n(\gamma)\cdot\fL_n(W\gamma),
\end{equation}
where
\begin{eqnarray}\label{frakMn}
&&\hspace{-0.7cm}\gM_n(\gamma):=\nonumber\\
&&\hspace{-0.7cm}\begin{pmatrix}
 D_{\gamma_1} & 0 & 0 & \ldots & 0 \\
 -\gamma_1\overline\gamma_2 & D_{\gamma_2} & 0 & \ldots & 0 \\
-\gamma_1D_{\gamma_2}\overline\gamma_3 & -\gamma_2\overline\gamma_3 & D_{\gamma_3} & \ldots & 0 \\
          \vdots & \vdots & \vdots & & \vdots \\
-\gamma_1\left(\prod_{j=2}^{n-1}D_{\gamma_j}\right)\overline\gamma_n & -\gamma_2\left(\prod_{j=3}^{n-1}D_{\gamma_j}\right)\overline\gamma_n & -\gamma_3\left(\prod_{j=4}^{n-1}D_{\gamma_j}\right)\overline\gamma_n & \ldots & D_{\gamma_n}
\end{pmatrix}.\nonumber\\
\end{eqnarray}
Moreover, $\gM_n(\gamma)$ is a nonsingular matrix which fulfills
\begin{equation}\label{MnForm}
I_n-\gM_n(\gamma)\gM_n^*(\gamma)=\eta_n(\gamma)\eta_n^*(\gamma),
\end{equation}
where
\begin{equation}\label{eta}
  \eta_n(\gamma)
   :=\left(\overline{\gamma_1},\overline{\gamma_2}D_{\gamma_1},\ldots,\overline{\gamma_n}
           \Bigg(\prod_{j=1}^{n-1}D_{\gamma_j}\Bigg)\right)^T.
\end{equation}
\end{lem}

\begin{cor}\label{cor14}
Let $\gamma=(\gamma_j)_{j=0}^\infty\in\Gamma l_2$ and let $n\in\N$. Then the multiplicative decomposition
\begin{equation}\label{LnDecomp}
\fL_n(\gamma)=\stackrel{\longrightarrow}{\prod_{k=0}^{\infty}}\gM_n(W^k\gamma)
\end{equation}
holds true.
\end{cor}

\begin{proof}
  Combine part (c) of Lemma \ref{lem12neu} and (\ref{LnProd}).
\end{proof}

Now we state the next main result of this paper.


For $h=(z_j)_{j=1}^\infty\in l_2$ and $n\in\N$ we set 
$$ h_n:=(z_1,\ldots,z_n)^\top\in\C^n .$$

\begin{thm}\label{s6-t2}
Let $\mu\in\cM_+^1(\T)$ and let $\gamma\in\Gamma$ be the sequence of Schur parameters associated with $\mu$.
Then $\mu$ is a Helson-Szeg\H o measure if and only if $\gamma\in\Gamma l_2$ and there exists some positive constant $C$
such that for all $h\in l_2$ the inequality
\begin{align}\label{s6-t2-2}
\lim\limits_{n\to\infty}\lim\limits_{m\to\infty}\left\|\Bigg(\stackrel\longleftarrow{\prod_{k=0}^m}\gM_n^*(W^k\gamma)\Bigg)h_n\right\|\ge C\|h\|
\end{align}
is satisfied.
\end{thm}

\begin{proof}
In view of \eqref{LnDecomp} and condition (c) in Lemma \ref{lem12neu} the condition \eqref{s6-t2-2}
is equivalent to the fact that for all $h\in l_2$ the inequality 
\begin{align}\label{s6-t2-3}
 \lim\limits_{n\to\infty}\|\cL_n^*(\gamma)h_n\|\ge C\|h\|
\end{align}
is satisfied. This inequality is equivalent to the inequality \eqref{s6-c1-1}.
\end{proof}


Theorem \ref{s6-t2} leads to an alternate proof of an interesting sufficient condition for a Szeg\H o measure to be a Helson-Szeg\H o measure
(see Theorem \ref{6-thm9}). To prove this result we will still need some preparations.

\begin{lem}\label{lem211}
Let $n\in\N$. Furthermore, let the nonsingular complex $n\times n$ matrix $\gM$ and the vector $\eta\in\C^n$ be chosen such that
\begin{equation}\label{lem211_36}
I_n-\gM\gM^*=\eta\eta^*
\end{equation}
holds. Then $1-\left\|\eta\right\|_{\C^n}^2>0$ and the vector
\begin{equation}\label{lem211_37}
\widetilde\eta:=\frac{1}{\sqrt{1-\left\|\eta\right\|_{\C^n}^2}}\,\gM^*\eta
\end{equation}
satisfies
\begin{equation}\label{lem211_38}
I_n-\gM^*\gM=\widetilde\eta\widetilde\eta^*.
\end{equation}
\end{lem}

\begin{proof}
The case $\eta=0_{n\times 1}$ is trivial. Now suppose that $\eta\in\C^n\backslash\{0_{n\times 1}\}$. From (\ref{lem211_36}) we get
\begin{equation}\label{lem211_40}
  (I_n-\gM\gM^*)\eta=\eta\eta^*\eta=\left\|\eta\right\|_{\C^n}^2\cdot\eta
\end{equation}
and consequently
\begin{equation}\label{lem211_41}
  \gM\gM^*\eta=(1-\left\|\eta\right\|_{\C^n}^2)\cdot\eta.
\end{equation}
Hence $1-\left\|\eta\right\|_{\C^n}^2$ is an eigenvalue of $\gM\gM^*$ with corresponding eigenvector $\eta$. Since $\gM$ is nonsingular, the matrix $\gM\gM^*$ is positive Hermitian. Thus, we have $1-\left\|\eta\right\|_{\C^n}^2>0$. Using (\ref{lem211_40}) we infer
\begin{equation}\label{lem211_43}
  (I_n-\gM^*\gM)\gM^*\eta=\gM^*(I_n-\gM\gM^*)\eta=\left\|\eta\right\|_{\C^n}^2\cdot\gM^*\eta.
\end{equation}
Taking into account (\ref{lem211_41}) we can conclude
\begin{equation}\label{lem211_46}
  \left\|\gM^*\eta\right\|_{\C^n}^2
   =\eta^*\gM\gM^*\eta
   =\eta^*\left[(1-\left\|\eta\right\|_{\C^n}^2)\cdot\eta\right]
   =(1-\left\|\eta\right\|_{\C^n}^2)\cdot\left\|\eta\right\|_{\C^n}^2
\end{equation}
and therefore from (\ref{lem211_37}) we have
\begin{equation}\label{lem211_44}
  \left\|\widetilde\eta\right\|_{\C^n} = \left\|\eta\right\|_{\C^n} >0.
\end{equation}
Formulas (\ref{lem211_43}), (\ref{lem211_37}) and (\ref{lem211_44}) show that $\left\|\widetilde\eta\right\|_{\C^n}^2$ is an eigenvalue of $I_n-\gM^*\gM$ with corresponding eigenvector $\widetilde\eta$. From (\ref{lem211_36}) and $\eta\neq0_{n\times 1}$ we get $$\rank(I_n-\gM^*\gM)=\rank(I_n-\gM\gM^*)=1.$$
So for each vector $h$ we can conclude
$$ (I_n-\gM^*\gM)h
  =(I_n-\gM^*\gM)  \mbox{$\big(h,\frac{\widetilde\eta \ }{ \ \left\|\widetilde\eta\right\|_{\C^n}}\big)_{_{\C^n}}
       \frac{\widetilde\eta \ }{ \ \left\|\widetilde\eta\right\|_{\C^n}}$}
  =(h,\widetilde\eta)_{_{\C^n}}\widetilde\eta=\widetilde\eta\widetilde\eta^* \cdot h. $$
\end{proof}


\begin{cor}\label{s6-c2}
Let the assumptions of Lemma \ref{lem211} be satisfied. Then for each $h\in\C^n$ the inequalities 
\begin{align}\label{s6-c2-1}
\|\gM h\|\ge(1-\|\eta\|^2)^{\frac12}\,\|h\|
\end{align}
and
\begin{align}\label{s6-c2-2}
\|\gM^*h\|\ge(1-\|\eta\|^2)^{\frac12}\,\|h\|
\end{align}
are satisfied.
\end{cor}

\begin{proof}
Applying \eqref{lem211_38} and \eqref{lem211_44} we get for $h\in\C^n$ the relation
\begin{align*}
\|h\|^2-\|\gM h\|^2=\big((I-\gM^*\gM)h,h\big)=\lvert(h,\widetilde\eta)\rvert^2\le\|\widetilde\eta\|^2\,\|h\|^2=\|\eta\|^2\,\|h\|^2.
\end{align*}
This implies \eqref{s6-c2-1}. Analogously, \eqref{s6-c2-2} can be verified.
\end{proof}

\begin{cor}\label{s6-c3}
Let $\gamma\in\Gamma l_2$, and let the matrix $\gM_n(\gamma)$ be defined via \eqref{frakMn}. Then for all $h\in\C^n$
the inequalities
\begin{align}\label{s6-c3-1}
\|\gM_n(\gamma)h\|\ge\Bigg(\prod_{j=1}^n D_{\gamma_j}\Bigg)\|h\|
\end{align}
and
\begin{align}\label{s6-c3-2}
\|\gM_n^*(\gamma)h\|\ge\Bigg(\prod_{j=1}^n D_{\gamma_j}\Bigg)\|h\|
\end{align}
are satisfied.
\end{cor}

\begin{proof}
	The matrix $\rklamFunk{ \fEl{\gM}{n} }{\gamma}$
	satisfies the conditions of Lemma \ref{lem211}. Here the
	vector $\eta$ has the form \eqref{eta}. It remains
	only to mention that in this case we have
	\begin{align}\label{s6-c3-3}
		1 - \norm{\eta}^2
	&=	1 - \abs{ \fEl{\gamma}{1} }^2
		  - \abs{ \fEl{\gamma}{2} }^2
		    \rklam{ 1 - \abs{ \fEl{\gamma}{1} }^2 }
		  - \ldots
		  - \abs{ \fEl{\gamma}{n} }^2
		    \eklam{ 
			\prod_{j = 1}^{n-1}{
				\rklam{1 - \abs{ \fEl{\gamma}{j} }^2}
			} 
		    }	\nonumber\\[5pt]
	&=	\prod_{j = 1}^{n}\rklam{
			1 - \abs{ \fEl{\gamma}{j} }^2
		}	\qPunkt
	\end{align}
\end{proof}

The above consideration lead us to an alternate proof for a nice
sufficient criterion for the Helson-Szeg\H{o} property of a measure
$\mu \in \rklamFunk{ \fEl{\cM^1}{+} }{\dT}$
which is expressed in terms of the modules of the associated
Schur parameter sequence.

\begin{thm}		\label{6-thm9}
	Let $\mu \in \rklamFunk{ \fEl{\cM^1}{+} }{\dT}$
	and let $\gamma	\in	\Gamma$ be the Schur parameter sequence
	associated with $\mu$. If $\gamma \in \Gamma \fEl{l}{2}$ and
	the infinite product
	\begin{align}		\label{6-24}
		\prod_{k = 1}^{\infty}{ 
			\prod_{j = k}^{\infty}{  
				\rklam{ 1 - \abs{ \fEl{\gamma}{j} }^2 }
			} 
		}
	\end{align}
	converges, then $\mu$ is a Helson-Szeg\H{o} measure.
\end{thm}

\begin{proof}
	Applying successively the estimate \eqref{s6-c3-2}, we
	get for all $m, \, n \in \N$ and all vectors $h \in \dC^n$
	the chain of inequalities
	\begin{align}
	&	\hspace{-18pt}
		\norm{  
			\eklam{
				\overleftarrow{
					\prod_{k = 0}^{m}
				}
				{
					\rklamFunk{ \fEl{\gM^*}{n} }{ W^k \gamma}
				}
			} h
		}	\nonumber	\\
	&=	\norm{  
			\rklamFunk{ \fEl{\gM^*}{n} }{ W^m \gamma}
			\eklam{
				\overleftarrow{
					\prod_{k = 0}^{m-1}
				}
				{
					\rklamFunk{ \fEl{\gM^*}{n} }{ W^k \gamma}
				}}
			 h
		}	\nonumber	\displaybreak[0]	\\[5pt]
	&\geq	\prod_{j = m+1}^{m+n}{ \fEl{D}{ \fEl{\gamma}{j} } }
		\norm{  
			\eklam{
				\overleftarrow{\prod_{k = 0}^{m-1}}
				{
					\rklamFunk{ \fEl{\gM^*}{n} }{ W^k \gamma}
				}
			} h
		}	
		\nonumber	\\[5pt]
	&\geq		\ldots\nonumber	\\[5pt]
	&\geq	\rklam{
			\prod_{j = m+1}^{m+n}{ \fEl{D}{ \fEl{\gamma}{j} } }
		}
		\cdot
		\rklam{
			\prod_{j = m}^{m+n-1}{ \fEl{D}{ \fEl{\gamma}{j} } }
		}
		\cdot
		\ldots
		\cdot
		\rklam{
			\prod_{j = 1}^{n}{ \fEl{D}{ \fEl{\gamma}{j} } }
		}
		\norm{h}	\nonumber	\displaybreak[0]	\\[5pt]
		&\geq	\rklam{
			\prod_{j = m+1}^{\infty}{ \fEl{D}{ \fEl{\gamma}{j} } }
		}
		\cdot
		\rklam{
			\prod_{j = m}^{\infty}{ \fEl{D}{ \fEl{\gamma}{j} } }
		}
		\cdot
		\ldots
		\cdot
		\rklam{
			\prod_{j = 1}^{\infty}{ \fEl{D}{ \fEl{\gamma}{j} } }
		}
		\norm{h}	\nonumber	\displaybreak[0]	\\[5pt]
	&=	\rklam{
			\prod_{k = 1}^{m+1}{
				\prod_{j = k}^{\infty}{ \fEl{D}{ \fEl{\gamma}{j} } }
			}
		}
		\norm{h}	\nonumber	\\[5pt]
	&\geq	\rklam{
			\prod_{k = 1}^{\infty}{
				\prod_{j = k}^{\infty}{ \fEl{D}{ \fEl{\gamma}{j} } }
			}
		}
		\norm{h}
	\end{align}
From this inequality it follows \eqref{s6-t2-2} where
$$ C=\prod_{k=1}^\infty\prod_{j=k}^\infty D_{\gamma_j} \;.$$
Thus, the proof is complete.
\end{proof}

Taking into account that the convergence of the infinite product \eqref{6-24}
is equivalent to the strong Szeg\H o condition
\begin{equation*}
  \sum_{k = 1}^{\infty}{ k \cdot \abs{ \fEl{\gamma}{k} }^2 } < \infty ,
\end{equation*}
Theorem \ref{6-thm9} is an immediate consequence of \cite[Theorem 5.3]{GKPY}.
The proof of \cite[Theorem 5.3]{GKPY} is completely different from the above
proof of Theorem \ref{6-thm9}. It is based on a scattering formalism using
CMV matrices (For a comprehensive exposition on CMV matrices, we refer the
reader to Chapter 4 in the monograph Simon \cite{Sim}.)

The aim of our next considerations is to characterize the Helson-Szeg\H o
property of a measure $\mu \in \rklamFunk{ \cM^{1}_{+} }{\dT}$ in terms
of some infinite series formed from its Schur parameter sequence. The
following result provides the key information for the desired characterization.

\begin{thm}		\label{6-thm13}
	Let $\gamma = \fKlam{\gamma}{j}{0}{\infty} \in \Gamma \fEl{\ell}{2}$ and let
	\begin{align}		\label{6-100}
		\rklamFunk{ \cA }{\gamma}
		:=  I - \rklamFunk{ \cL }{\gamma} \rklamFunk{ \Adj{\cL} }{\gamma}
	\end{align}
	where $\rklamFunk{ \cL }{\gamma}$ is given by \eqref{5-thm7-1}.
	Then $\rklamFunk{ \cA }{\gamma}$ satisfies the inequalities
	\begin{align}		\label{6-101}
		0 \leq 
		\rklamFunk{ \cA }{\gamma}
		\leq I
	\end{align}
	and admits the strong convergent series decomposition
	\begin{align}		\label{6-102}
		\rklamFunk{ \cA }{\gamma}
		= \sum_{j = 0}^{\infty}{ 
		    \rklamFunk{ \fEl{ \xi }{j} }{ \gamma } 
		    \rklamFunk{ \fEl{ \Adj{\xi} }{j} }{ \gamma }
		  }
	\end{align}
	where
	\begin{align}		\label{6-103}
		 \rklamFunk{ \fEl{ \xi }{0} }{ \gamma } 
		 :=  \rklamFunk{ \eta }{ \gamma } ,
	 \qquad
	 \qquad
	     \rklamFunk{ \fEl{ \xi }{j} }{ \gamma } 
	    := \eklam{
				\overrightarrow{
					\prod_{k = 0}^{j - 1}
				}
				{
					\rklamFunk{\gM}{ W^k \gamma}
				}
			 }
			 \rklamFunk{ \eta }{ W^j \gamma } ,
			 j \in \N,
	\end{align}
	and $\rklamFunk{\gM}{\gamma}$, $\rklamFunk{\eta}{\gamma}$
	and $W$ are given by \eqref{s6-2A-1B}, \eqref{s6-2A-1D} and \eqref{coshift},
 respectively.
\end{thm}

\begin{proof}
  Since the matrix $\rklamFunk{ \cL }{\gamma}$ is a block of the unitary
  operator matrix given by \eqref{5-20} we have
  $\norm{ \rklamFunk{ \cL }{\gamma} } \leq 1$.
  This implies the inequalities \eqref{6-101}. Using \eqref{s6-2A-1B}
  and \eqref{s6-2A-1C}, we obtain
	\begin{align*}
		\rklamFunk{ \cA }{\gamma}
	&	= I - \rklamFunk{ \cL }{\gamma} 
	         \rklamFunk{ \Adj{\cL} }{\gamma} \displaybreak[0]  \\[5pt]
	&	= I - \rklamFunk{\gM}{\gamma}
	         \rklamFunk{ \cL }{W \gamma} 
	         \rklamFunk{ \Adj{\cL} }{W \gamma}
	         \rklamFunk{ \Adj{\gM} }{\gamma} \displaybreak[0]  \\[5pt]
	&	= I - \rklamFunk{\gM}{\gamma}  \rklamFunk{ \Adj{\gM} }{\gamma}
	       + \rklamFunk{\gM}{\gamma}
	         \rklamFunk{ \cA }{W \gamma} 
	         \rklamFunk{ \Adj{\gM} }{\gamma}\displaybreak[0]  \\[5pt]
	&	= \rklamFunk{\eta}{\gamma}
	     \rklamFunk{ \Adj{\eta} }{\gamma}
	       + \rklamFunk{\gM}{\gamma}
	         \rklamFunk{ \cA }{W \gamma} 
	         \rklamFunk{ \Adj{\gM} }{\gamma} .
	\end{align*}
   Repeating this procedure $m-1$ times, we get
	\begin{align*}
		\rklamFunk{ \cA }{\gamma}
	&	= \rklamFunk{\eta}{\gamma}
	     \rklamFunk{ \Adj{\eta} }{\gamma}
	     + \rklamFunk{\gM}{\gamma}
	         \rklamFunk{ \eta }{W \gamma} 
	         \rklamFunk{ \Adj{\eta} }{W \gamma}
	         \rklamFunk{ \Adj{\gM} }{\gamma}  \\[2pt]
	&   \qquad   + \ldots + 
	     \eklam{
				\overrightarrow{
					\prod_{k = 0}^{m - 1}
				}
				{
					\rklamFunk{\gM}{ W^k \gamma}
				}
			 }
			 \rklamFunk{\eta}{W^m \gamma}
	       \rklamFunk{ \Adj{\eta} }{W^m \gamma}
			\eklam{
				\overleftarrow{
					\prod_{k = 0}^{m - 1}
				}
				{
					\rklamFunk{ \Adj{\gM} }{ W^k \gamma}
				}
			 }  \\[2pt]
   &   \qquad    \qquad +
	     \eklam{
				\overrightarrow{
					\prod_{k = 0}^{m - 1}
				}
				{
					\rklamFunk{\gM}{ W^k \gamma}
				}
			 }
			 \rklamFunk{\cA}{W^{m + 1} \gamma}
			\eklam{
				\overleftarrow{
					\prod_{k = 0}^{m - 1}
				}
				{
					\rklamFunk{ \Adj{\gM} }{ W^k \gamma}
				}
			 }\displaybreak[0]  \\[5pt]
	&  = \sum_{j = 0}^{m - 1}{ 
		    \rklamFunk{ \fEl{ \xi }{j} }{ \gamma } 
		    \rklamFunk{ \fEl{ \Adj{\xi} }{j} }{ \gamma }
		  }
		  + \eklam{
				\overrightarrow{
					\prod_{k = 0}^{m - 1}
				}
				{
					\rklamFunk{\gM}{ W^k \gamma}
				}
			 }
			 \rklamFunk{\cA}{W^{m + 1} \gamma}
			\eklam{
				\overleftarrow{
					\prod_{k = 0}^{m - 1}
				}
				{
					\rklamFunk{ \Adj{\gM} }{ W^k \gamma}
				}
			 } \ldotp
	\end{align*}
	In view of part (c) of Lemma \ref{lem11a} and the shape \eqref{s6-2A-1B}
	of the matrix $\rklamFunk{\gM}{\gamma}$ for finite vectors $h \in \fEl{\ell}{2}$
	(i.e. $h$ has the form
	$h = \rklamFunk{{\rm col}}{ \fEl{z}{1}, \, \fEl{z}{2}, 
	       \, \ldots \, \fEl{z}{n}, \, 0, \, 0, \, \ldots}$ for some $n \in \N$)
	we obtain
	\begin{equation*}
	   \lim_{m \longrightarrow \infty}{
	   \eklam{
				\overrightarrow{
					\prod_{k = 0}^{m - 1}
				}
				{
					\rklamFunk{\gM}{ W^k \gamma}
				}
			 }
			 \rklamFunk{\cA}{W^{m + 1} \gamma}
			\eklam{
				\overleftarrow{
					\prod_{k = 0}^{m - 1}
				}
				{
					\rklamFunk{ \Adj{\gM} }{ W^k \gamma}
				}
			 }
	   } h
	   = 0 .
	\end{equation*}
	This implies that the series given by the right-hand side of the formula
	\eqref{6-103} weakly converges to $\rklamFunk{\cA}{\gamma}$. From the
	concrete form of this series, its strong convergence follows.
	Thus, the proof is complete.
\end{proof}

The last main result of this paper is the following statement, which is an
immediate consequence of Theorem \ref{s6-t1} and Theorem \ref{6-thm13}.

\begin{thm}		\label{6-thm14}
	Let $\mu \in \rklamFunk{ \cM^{1}_{+} }{\dT}$ and let $\gamma \in \Gamma$
	be the sequence of Schur parameters associated with $\mu$. Then $\mu$ is a
	Helson-Szeg\H o measure if and only if $\gamma \in \Gamma \fEl{\ell}{2}$
	and there exists some positive constant $\varepsilon \in \rklam{0, \, 1}$ such
	that the inequality
	\begin{align}		\label{6-104}
		\sum_{j = 0}^{\infty}{ 
		    \rklamFunk{ \fEl{ \xi }{j} }{ \gamma } 
		    \rklamFunk{ \fEl{ \Adj{\xi} }{j} }{ \gamma }
		  }
		  \leq \rklam{ 1 - \varepsilon} I
	\end{align}
	is satisfied, where the vectors
	$\rklamFunk{ \fEl{ \xi }{j} }{ \gamma }, \, j \in \N_0$, are given by
	\eqref{6-103}.
\end{thm}

We note that the inequality \eqref{6-104} can be considered as a rewriting
of condition \eqref{s6-t2-2} in an additive form.

\begin{rem}   \label{6-rem15}
  Finally, we would like to add that many important properties
  of Schur functions can be characterized in terms
  of the matrix, $\rklamFunk{ \cL }{\gamma}$, given by \eqref{5-thm7-1}.
  It was shown in \cite[Section 5]{7} that the pseudocontinuability of
  a Schur function is determined by the properties of the matrix
  $\rklamFunk{ \cL }{\gamma}$. In \cite[Section 2]{9}, it was proved
  that the $S$-recurrence property of Schur parameter sequences of
  non-inner rational Schur functions is also expressed with the aid of
  the matrix $\rklamFunk{ \cL }{\gamma}$. Furthermore, the structure of
  the matrix $\rklamFunk{ \cL }{\gamma}$ allows one to determine whether
  a non-inner Schur function is rational or not(see (\cite[Section 2]{9}).
\end{rem}


\end{document}